\documentclass[10pt]{amsart}

\usepackage{times,amsfonts,amsmath,amstext,amsbsy,amssymb,
amsopn,amsthm,upref,eucal, amscd}

\usepackage[colorlinks=true]{hyperref}
\usepackage[T1]{fontenc}
\usepackage{multirow,hhline,longtable}
\usepackage{marginnote}
\usepackage{color,tcolorbox,colortbl}
\usepackage[displaymath, mathlines, pagewise]{lineno}

\usepackage{graphicx}
\usepackage{algpseudocode,algorithm}
\usepackage{caption}
\usepackage[skip=0.5ex]{subcaption}

\newtheorem{theorem}{Theorem}[section]
\newtheorem{lemma}[theorem]{Lemma}

\numberwithin{equation}{section}

\theoremstyle{definition}

\newtheorem{example}[theorem]{Example}

\newtheorem{remark}[theorem]{Remark}

\def\leq{\leqslant }

\makeatletter
\newcommand{\@minipagerestore}{\setlength{\parskip}{\medskipamount}
\setlength{\parindent}{12pt}}
\makeatother

\usepackage{accents}

\newcommand\intpx[2]{S_{(#1;#2)}}

\newcommand\pvlong[1]{\textcolor{black}{#1} }
\newcommand\pvshort[1]{\textcolor{black}{#1} }

\newcommand*\Let[2]{\State #1 $\gets$ #2}
\algrenewcommand\algorithmicrequire{\textbf{Input:}}
\algrenewcommand\algorithmicensure{\textbf{Output:}}

\algrenewcommand\alglinenumber[1]{ \sf\scriptsize{#1}}
\newcommand{\nocontentsline}[3]{}
\newcommand{\tocless}[2]{\bgroup\let\addcontentsline=\nocontentsline#1{#2}\egroup}

\begin{document}

\title[The graph of the dimension function of the classical spectra]{On the graph of the dimension function of the Lagrange and Markov spectra} 

\author[C. Matheus, C. G. Moreira and P. Vytnova]{Carlos Matheus,
Carlos Gustavo Moreira and Polina Vytnova} 

\address{C. Matheus, CNRS \& \'Ecole Polytechnique, CNRS (UMR 7640), 91128, Palaiseau, France.}
\email{matheus.cmss@gmail.com}

\address{C. G. Moreira, School of Mathematical Sciences, Nankai University, Tianjin 300071, P. R.China, and IMPA, Estrada Dona Castorina 110, CEP 22460-320, Rio de Janeiro, Brazil.}
\email{gugu@impa.br}

 \address{P. Vytnova, Department of Mathematics, University of Surrey,
 Guildford, GU2 7XH, UK.}
\email{P.Vytnova@surrey.ac.uk}

\thanks{The second author was partly supported by CNPq, FAPERJ, and INCTMAT project of J. Palis. The third author was partly supported by EPSRC grant EP/T001674/1. }

\date{\today}

\begin{abstract}
    We study the graph of the function $d(t)$ encoding the Hausdorff dimensions
    of the classical Lagrange and Markov spectra with half-infinite lines of the
    form $(-\infty, t)$. For this sake, we use the fact that the Hausdorff
    dimension of dynamically defined Cantor sets drop after erasing an element of its
    Markov partition to determine twelve non-trivial plateaux of $d(t)$. Next,
    we employ rigorous numerical methods (from our recent joint paper with
    Pollicott) to produce approximations of the graph of $d(t)$ between these
    twelve plateaux. As a corollary, we prove that the largest ten non-trivial
    plateaux of $d(t)$ are exactly those plateaux with lengths $>0.005$. 
\end{abstract}\maketitle


\section{Introduction}
\label{s:intro}
The best constants of Diophantine approximations of irrational numbers and real
indefinite binary quadratic forms constitute two subsets of the real line known
as the Lagrange and Markov spectra. More concretely, the \emph{Lagrange
spectrum} is defined by
$$L:=\left\{\limsup\limits_{\substack{p,q\to\infty \\ p\in\mathbb{Z},
q\in\mathbb{N}}}\frac{1}{|q^2\alpha-pq|}<\infty: \alpha\in\mathbb{R}\right\}$$ 
and the \emph{Markov spectrum} is defined by
$$M:=\left\{\sup\limits_{\substack{(p,q)\in\mathbb{Z}^2 \\
(p,q)\neq(0,0)}}\frac{1}{|ap^2+bpq+cq^2|}<\infty: ax^2+bxy+cy^2 \textrm{ real
indefinite, }b^2-4ac=1\right\}.$$  

The features of these sets were systematically by several authors since the
foundational works by Markov~\cite{Mar79},~\cite{Mar80} in 1879--1880, and we
refer the reader to the book of Cusick--Flahive~\cite{CF} for a nice account on the literature on this topic before 1989. 

The Lagrange and Markov spectra possess fascinating fractal structures between $3$ and the so-called \emph{Freiman's constant} $c_F=4.5278\dots$. In particular, the second author proved in \cite{Mor18} that 
$$\dim(L\cap(-\infty,t)) = \dim(M\cap(-\infty,t))$$ 
for all $t\in\mathbb{R}$ (where $\dim$ stands for the Hausdorff dimension) and, moreover, the \emph{dimension function} 
$$d(t):=\dim(L\cap(-\infty,t)) = \dim(M\cap(-\infty,t))$$ 
is a \emph{continuous} function such that $d(3+\varepsilon)>0$ and
$d(\sqrt{12})=1$. Furthermore, it was recently shown in~\cite{MMPV}
that\footnote{Throughout the manuscript all numbers are truncated and not rounded.} 
$$t_1:=\min\{t\in\mathbb{R}: d(t)=1\}=3.334384\dots$$ 

In this note, we investigate the graph of the dimension function $d(t)$. More
concretely, it is known that $L\subset M$ are closed subsets of the real line,
so that their complements are countable unions of open intervals consisting of
their gaps. By definition, the dimension function $d(t)$ is constant on the gaps
of $M$, that is, any such gap is included in a \emph{plateau} of $d(t)$. Hence,
we can \emph{rigorously} depict a portion of the graph of $d(t)$ by determining
the first few largest plateaux of $d(t)$ between $3$ and $t_1$.
\pvshort{The plot of the function~$d$ on the complement to the plateaux can be found in
Appendix~\ref{ap:plot}. }

For the smoothness of exposition, let us denote the two infinite
plateaux by $P_0 = (-\infty,3)$ and $P_1 = (t_1, + \infty)$. 
Our main result, Theorem~\ref{t.A} below, states that the ten largest (non-trivial)
plateaux $P_2, P_3,\dots, P_{11}$ of $d(t)$ are precisely the plateaux of
length~$0.005$. As a byproduct, we also identify two plateaux of
length~$0.004185$ and~$0.004003$ and this helps to reduce necessary numerical
computations. Interestingly enough, the numbers and sizes of the gaps of $M$ included in these plateaux are variable: for
instance,  
\begin{itemize} 
    \item the largest plateau $P_2$ is the closure of the union of the largest gap $G_1=(a_1,b_1)$ of $M$ before $t_1$ and another gap $J_1=(b_1,c_1)$ of $M$, where $b_1$ is an isolated point of $L$; 
    \item the \emph{second} largest plateau $P_3$ is the closure of the union of the \emph{third} largest gap $G_3=(a_3,b_3)$ of $M$ before $t_1$ and another gap $J_3=(b_3,c_3)$ of $M$, where $b_3$ is an isolated point of $L$; 
    \item the \emph{third} largest plateau $P_4$ is the closure of the union of the \emph{second} largest gap $G_2=(a_2,b_2)$ of $M$ before $t_1$ and another gap $J_2=(b_2,c_2)$ of $M$, where $b_2$ is an isolated point of $L$; 
    \item the seventh largest plateau $P_8$ is a \emph{single} gap of $M$. 
\end{itemize} 

\begin{minipage}{80mm}
\begin{remark} 
   In general, the structure of gaps of the spectra contained in small plateaux of $d(t)$ can be quite intricate: for instance, this is the case of any plateau of $d(t)$ containing a gap of $L$ including a Cantor set of points of $M\setminus L$. 
\end{remark}

 The main goal of this note is to rigorously justify the drawing of the
dimension function~$d(t)$ of the Lagrange and Markov spectra shown in
Figure~\ref{fig:functionf}. To this end, we first identify twelve plateaux. Then
we compute an approximation to the function on the complement to the plateaux using the same
method as was used in~\cite{MMPV} to compute the value of~$t_0$. It is rooted
in the approach introduced by Bumby in~\cite{B74} and relies on two Lemmas
that connect the value $d(t)$ to the dimension of a certain Gauss--Cantor sets.
\\
\indent The dimension of the arising Gauss--Cantor sets is computed using 
the method for computing the Hausdorff dimension originally developed in~\cite{PV20} 
and then improved in~\cite{MMPV}. 
\\
\indent We organise this paper as follows. In \S\ref{s.prelim}, we briefly review some
classical facts about the Lagrange and Markov spectra, including their dynamical
characterisations due to Perron. In \S\ref{s.plateaux}, we state and prove the
main result of this note, namely, Theorem~\ref{t.A}, about the first few
(non-trivial) plateaux of the dimension function $d(t)$.
\end{minipage}
\begin{minipage}{80mm}
   \includegraphics{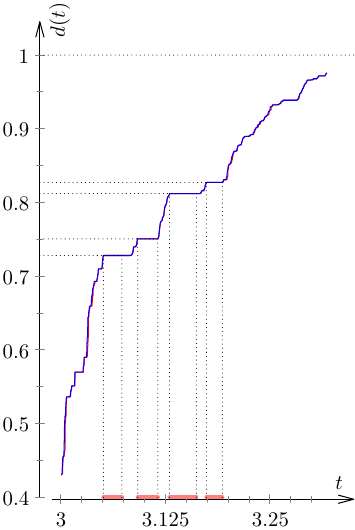}
    \captionof{figure}{A plot of the function~$d$. Pink intervals are gaps
     in the Markov spectrum identified in~\cite{CF}}
    \label{fig:functionf}
\end{minipage} 

Finally, in \S\ref{s.approx}, we discuss the numerical methods
yielding rigorous approximations to the graph of~$d(t)$. \pvshort{The numerical results
can be found in the Appendix. 
The code is available at the third author's github
\href{https://github.com/Polevita/Dimension\_Function}{https://github.com/Polevita/}.
}


\subsection*{Acknowledgments} 
The authors are very grateful to Prof. David Loeffler for providing
access to one of the computer servers that belong to the Number
Theory group at the University of Warwick. All computations have been carried
out on that machine. 

\section{Preliminaries}\label{s.prelim}

\subsection{Perron's characterisation of the classical spectra} Let $(\mathbb N^*)^\mathbb Z$ be the set of bi-infinite sequences of elements of $\mathbb{N}^*=\mathbb{N}\setminus\{0\}$. Given $\underline{\alpha} = (\alpha_n)_{n\in\mathbb{Z}} \in (\mathbb N^*)^\mathbb Z$ and $k\in\mathbb{Z}$, we set 
$$\lambda_k(\underline{\alpha}):=[\alpha_k;\alpha_{k+1},\alpha_{k+2},\ldots]+[0;\alpha_{k-1},\alpha_{k-2},\ldots],$$ 
where $[b_0;b_1,b_2,\ldots]$ stands for the continued fraction expansion 
$$[b_0;b_1,b_2,\ldots]:=b_0+\cfrac{1}{b_{1}+\cfrac{1}{b_{2}+\cfrac{1}{\ddots}}}.$$ 
Let $\sigma:(\mathbb N^*)^{\mathbb Z}\to (\mathbb N^*)^\mathbb Z$, $\sigma( (\alpha_n)_{n \in \mathbb
Z} ) = (\alpha_{n+1})_{n\in \mathbb Z}$, be the Bernoulli shift. The Lagrange value $\underline{\alpha}$ is the
limit superior of values of $\lambda_0$ along the $\sigma$-orbit of $\underline{\alpha}$, i.e.,  
$$
\ell(\underline{\alpha}) := \limsup_{n\to \infty} \lambda_0(\sigma^n \underline{\alpha})
=\limsup_{n\to \infty} \lambda_n(\underline{\alpha})
$$
and the Markov value of~$\underline{\alpha}$ is the supremum of values of $\lambda_0$ along
the $\sigma$-orbit of $\underline{\alpha}$, i.e.,  
$$
m(\underline{\alpha}) : =  \sup_{n\in \mathbb Z} \lambda_0(\sigma^n \underline{\alpha}) =
\sup_{n\in \mathbb Z}\lambda_n(\underline{\alpha}).
$$ 
In 1921, Perron~\cite{Per21} showed that the classical spectra are the
collections of (finite) Lagrange and Markov values: 
\begin{equation}
    \label{eq:LMdef}
    L := \left\{ \ell(\alpha)\in\mathbb{R} \mid \alpha \in
    (\mathbb N^*)^\mathbb Z \right\} \quad \mbox{ and } \quad
    M := \left\{  m(\alpha)\in\mathbb{R} \mid \alpha \in (\mathbb N^*)^\mathbb Z \right\}. 
\end{equation}

\subsection{Some known facts about the classical spectra} 
\label{ss.dimension-formula} 
Perron's description of $L$ and $M$ allows to check that $L\subset M$ are closed
subsets of $\mathbb{R}$. As it was shown by Markov (circa~$1880$), 
$$
L\cap(-\infty,3) = M\cap (-\infty, 3) = \{\sqrt{5}<\sqrt{8}<\sqrt{221}/5<\dots\}
$$
is an explicit sequence of quadratic irrationals converging to $3$. 
Also, Perron
showed in 1921 that $\sqrt{12}, \sqrt{13}\in L$,  

\begin{minipage}{90mm}
$$
L\cap(\sqrt{12},\sqrt{13})=M\cap(\sqrt{12},\sqrt{13}) = \varnothing, 
$$ 
and $m(\underline{\alpha})\leq\sqrt{12}$ if and only if
$\underline{\alpha}\in\{1,2\}^{\mathbb{Z}}$. Moreover, Hall~\cite{Hall47} showed
in 1947 that there is a smallest constant $c_F\leq 6$ such that
$L\cap[c_F,\infty)=M\cap[c_F,\infty) = [c_F,\infty)$, and Freiman~\cite{Fr75}
determined the exact value of $c_F$ in~$1975$. 

\pvlong{More recently, as we mentioned in \S\ref{s:intro}, Moreira proved in 2018 that 
$
d(t)=\dim(L\cap(-\infty,t)) = \dim(M\cap(-\infty,t))
$ 
is a continuous function of $t$ with $d(3+\varepsilon)>0$ for all
$\varepsilon>0$. Finally, in~\cite{EMGR22} it was shown that near~$3$ the
dimensions of Markov and Lagrange spectra have asymptotic
$$
d\left(3+\varepsilon\right)=2\cdot \frac{W(c|\log\varepsilon|)}{|\log\varepsilon|}
+O\left(\frac{\log|\log\varepsilon|}{|\log\varepsilon|^2}\right), 
$$
where~$W$ is the Lambert function and $c = \frac{1}{\log(3+\sqrt5)/2}$.
Figure on the right shows the asymptotic estimate (smooth solid black line) 
in comparison with the lower and upper bound on the dimension function~$d$. }
Furthermore, Moreira also proved in \cite{Mor18} the following useful dimension formula: 
$
d(t)=\min\{1, 2 \cdot D(t)\},
$ 
where $D(t)=\dim\bigl(\bigl\{[0;a_1,\dots, a_n,\dots] \mid
\exists a_{-j}\in{\mathbb N}^*, j\in\mathbb N, [a_k;a_{k+1},\dots]+[0;a_{k-1},a_{k-2},\dots]\leq t,
\,\,\forall\,k\in\mathbb{Z}\bigr\}\bigr)$. 
\end{minipage}%
\begin{minipage}{75mm}
   \includegraphics{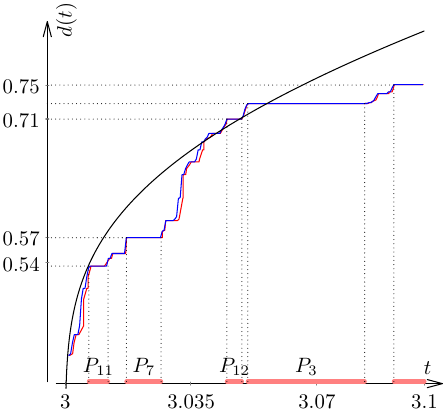}
    \captionof{figure}{A plot of the function~$d$ and its asymptotic near~$3$. Pink intervals are the
    plateaux identified in Theorem~\ref{t.A}.} 
    \label{fig:lambert}
\end{minipage}
\begin{remark}
It follows from the previous facts that function $d(t)$ is a Cantor staircase
function, i.e., it is a non-constant and non-decreasing continuous function which is locally
constant in an open and dense set. Indeed, $d(t)=1$ for every $t\ge t_1$ and,
for every positive integer $n$,
$d(t_1-1/n)=\dim(L\cap(-\infty,t_1-1/n))=d(t_1-1/n)<1$, and in particular
$L\cap(-\infty,t_1-1/n)$ has Lebesgue measure zero. This implies that $L\cap
(-\infty,t_1)=\cup_{n\ge 1}L\cap(-\infty,t_1-1/n)$ also has Lebesgue measure
zero, and therefore has empty interior. So, the set $L\cap (-\infty,t_1]$ is
closed, has Lebesgue measure zero and empty interior, and the function $d(t)$ is
constant in every connected component of the open, dense and full measure set
$\mathbb R\setminus L\cap (-\infty,t_1]$. 
\end{remark}
\subsection{A useful lemma about continued fractions} In \S\ref{s.plateaux} below, the following lemma will be used several times:

\begin{lemma}
    \label{l.plateaux} 
    Let $\beta$ be a finite word and $\theta$ be a symmetric finite word of even size on $\{1,2\}$. Assume that, under some conditions (e.g., after forbidding a finite list of finite strings), the Markov value of an infinite word of the type $\gamma=\omega_2^t \beta^t 2\theta 2\beta \omega_1$ is attained at one of the $2$ next to $\theta$, where $\omega_1, \omega_2$ are infinite words (on $\{1,2\}$). Then this Markov value is minimum when $\omega_1=\omega_2=:\omega$ and $[0;\beta,\omega]$ is minimum (under these conditions).
\end{lemma} 

\begin{proof} Let $x=[0;\beta,\omega_1]$ and $y=[0;\beta,\omega_2]$. If the positions with the $2$'s after and before $\theta$ are, respectively, $0$ and $2r+1$ (where $2r=|\theta|$), then $\lambda_0(\gamma)=2+y+[0;\theta,2+x]=:h(x,y)$ and $\lambda_{2r+1}(\gamma)=2+x+[0;\theta,2+y]=h(y,x)$. Assume without loss of generality (by symmetry) that the Markov value of $\gamma$ is attained at the position $0$. Since $h(x,y)$ is increasing in $y$ and decreasing in $x$, then we should have $y\ge x$ (indeed if $y<x$ then $\lambda_0(\gamma)=h(x,y)<h(x,x)<h(y,x)=\lambda_{2r+1}(\gamma)$, a contradiction). 
Consequently, in this case we have $h(x,y)\ge h(y,y)$, so the minimum value in this case is attained when $x=y$, and is equal to $h(y,y)=2+y+[0;\theta,2+y]$. Since $f(y)=[0;\theta,2+y]$ is a contraction, $h(y,y)$ is increasing, so the minimum is attained when $y=[0;\beta,\omega_2]=[0;\beta,\omega]$ is minimum. 
\end{proof} 

\subsection{A useful lemma about Hausdorff dimension} Let $K$ be a $C^2$ dynamical Cantor set, i.e., 
$$K=\bigcap\limits_{n\in\mathbb{N}}\psi^{-n}(I_1\cup\dots\cup I_r),$$ 
where $\psi:I_1\cup\dots\cup I_r\to \mathbb{R}$ is a topologically mixing,
piecewise expanding $C^2$ map whose domain is a disjoint union of compact
intervals forming a Markov partition for $\psi$. The following folklore lemma
(see, e.g., Lemma 2.5 of Lima--Moreira~\cite{LM}) says that the removal of an element of the construction of a dynamical Cantor set strictly decreases the Hausdorff dimension. 

\begin{lemma}\label{l.dimension-drop} Let $K$ be a dynamical Cantor set defined by $\psi:I_1\cup\dots\cup I_m\to\mathbb{R}$. Given $m\in\mathbb{N}$, let $\mathcal{P}^m$ be the collection of connected components of $\psi^{-m+1}(I_1\cup\dots\cup I_r)$. If $J\in\mathcal{P}^m$ is an interval of the construction of $K$ such that $\mathcal{P}^m\setminus\{J\}$ is a Markov partition with respect to $\psi$, then the dynamical Cantor set 
$$K'=\bigcap\limits_{n\in\mathbb{N}} \psi^{-nm}\left(\bigcup\limits_{I\in\mathcal{P}^m\setminus\{J\}}I\right)$$ 
has a Hausdorff dimension $\textrm{dim}(K') < \textrm{dim}(K)$. 
\end{lemma} 

For our purposes, we shall apply this lemma to the so-called
\emph{Gauss--Cantor sets}, that is, the dynamical Cantor sets produced by the
restrictions of iterates of the Gauss map
$g([0;a_1,a_2,a_3,\dots]):=[0;a_2,a_3,\dots]$ to an appropriate finite
collection of compact subintervals of $(0,1)$. It is convenient to describe them
in terms of the so-called \emph{forbidden subsequences} of the
sequence~$(a_j)_{j\in \mathbb N}$.

\section{Plateaux of the function $d(t)$}\label{s.plateaux} 

The largest plateaux of the dimension function of the classical spectra are described by the following statement. 
\begin{theorem}~\label{t.A}
The ten largest plateaux of the dimension function~$d$ are precisely the
plateaux of length bigger than~$0.005$. 
\end{theorem}
The proof is constructive and consists of two parts. First, we identify the ten
largest plateaux and two extra auxiliary plateaux. Then, estimating~$d$
numerically, we show that there is no other plateaux longer than~$0.005$. In
order to specify the plateaux, we use Gauss--Cantor sets that we would like
to introduce now. \\ 
\indent Consider the sets~$FW_i$ of forbidden subsequences given in the table
below:

\smallskip

    \begin{tabular}{|cl||cl||cl|}
        \hline
        $i$ & $FW_i$  & $i$ & $FW_i$ & $i$ & $FW_i$ \\
        \hline
        $2$ & $121$ &                               $6$ & $1212$, $2121$ &         $10$ & $1212$, $2121$, $1112111$  \\
        $3$ & $121$, $212$ &                        $7$ & $121$, $212$, $111222$, $222111$ &    $11$ & $121$, $212$, $111222$, $222111$, $21112$  \\
        $4$ & $121$, $21222$, $22212$ &             $8$ & $1212$, $2121$, $211121112$ &         $12$ & $121$, $212$, $12221112$, $21112221$ \\
        $5$ & $1212$, $2121$, $12111$, $11121$ &    $9$ & $21212$, $111212$, $212111$ &         $13$ & $21212$, $111212$, $212111$, $12121122$, $22112121$ \\
        \hline
\end{tabular}\\ \\
Let us define the Gauss--Cantor sets of continued fractions by 
$$
K_i = \{[0;a_1,a_2,\dots]: (a_n)_{n\in\mathbb{N}}\in\{1,2\}^{\mathbb{N}}
\mbox{ doesn't contain a sequence from } FW_i \mbox{ as a subsequence} \},
\ i = 2,3, \ldots, 13;
$$ 
and let $D_i=\dim(K_i)$, $i=2,\dots, 13$, be the Hausdorff dimensions of 
these Gauss--Cantor sets.

\noindent We now can list the twelve (non-trivial) plateaux of $d(t)$ in
decreasing order of lengths, that we were able to identify.  
\begin{align*}
P_2&=d^{-1}(2D_2) = \left[\frac{4\sqrt{30}}7,\frac{16351+720\sqrt{30}}{6409}\right] = 
[3.129843\ldots, 3.166578\ldots]; \\
P_3&=d^{-1}(2D_3) = \left[\frac{4\sqrt{210}}{19},\frac{74613+4096\sqrt{210}}{43449}\right]=
[3.050816\ldots, 3.083377\ldots]; \\
P_4&=d^{-1}(2D_4) =
\left[\frac{58473-745\sqrt{210}}{15422},\frac{1341491634738+14336533\sqrt{210}}{430566005892}\right] =
[3.091487\dots,3.116129\dots];\\
P_5&=d^{-1}(2D_5) =
\left[\frac{15-\sqrt{30}}{3},\frac{332001+3760\sqrt{30}}{110409}\right] =
[3.174258\dots, 3.193538\dots]; \\ 
P_6&=d^{-1}(2D_6) = \left[\frac{4\sqrt{6}}{3},\frac{18879+800\sqrt{30}}{7089}\right] = 
[3.265986\dots,3.281249\dots];\\
P_7&=d^{-1}(2D_7) = \left[\frac{28\sqrt{213378}}{4287}, \frac{2497149255+2842763\sqrt{213378}}{1258910718}\right] = 
[3.017028\dots, 3.026666\dots];\\
P_8&=d^{-1}(2D_8) = \left[\frac{\sqrt{66045}}{79},\frac{190+4\sqrt{30}}{65}\right] = 
[3.253066\dots, 3.260136\dots];\\
P_9&=d^{-1}(2D_9) = \left[\frac{11+35\sqrt{87}}{102},\frac{56508+2716\sqrt{87}}{24687}\right] = 
[3.308414\dots,3.315152\dots];\\
P_{10}&=d^{-1}(2D_{10}) = \left[\frac{14775-599\sqrt{30}}{3570},\frac{9537392579+1230099\sqrt{30}}{2959418678}\right] = 
[3.219647\dots,3.225001\dots];\\
P_{11}& = d^{-1}(2D_{11}) = \left[\frac{10\sqrt{718341}}{2819},
\frac{14264157401+16294182 \sqrt{718341}}{9321104530}\right] =
[3.006562\dots, 3.011906\dots];\\
P_{12}& = d^{-1}(2D_{12})= \left[\frac{8\sqrt{1785}}{111},\frac{89042158285+148687392\sqrt{1785}}{31262231449} \right] =
[3.044991\dots, 3.049177\dots]; \\
P_{13}& = d^{-1}(2D_{13}) = \left[\frac{24\sqrt{35}}{43}, \frac{457878845-1713407 \sqrt{87}}{133663998}\right] = 
[3.301998\dots, 3.306030\dots].\\
\end{align*}

\begin{remark} 
    The numerical values of $D_i$ are specified in Appendix~\ref{ap:table}. 
    We would like to highlight that the plateaux do not occur in the order of
    increasing length.     
\end{remark}


The remainder of this section is entirely dedicated to the
proof of this theorem \emph{modulo} some rigorous Hausdorff dimension estimates explained in \S\ref{s.approx}. More precisely, each subsection of \S\ref{ss:plateaux-list} below discusses one of
the twelve plateaux above. Afterwards, in \S\ref{ss:inbetween}, we complete the
proof of this theorem (modulo some rigorous numerical computations explained in \S\ref{s.approx}) by showing that any non-trivial plateaux of $d(t)$
distinct from $P_i$, $2\leq i\leq 13$, has length $<0.005$. 

\subsection{The twelve plateaux}\label{ss:plateaux-list} 

\subsubsection{The largest plateau: first appearances of $121$} 

Cusick identified in~\cite{C74} a gap in Markov (and Lagrange) spectrum between
$[2; \overline{1,2,2,2}]+[0;
\overline{2,2,1,2}]=\frac{4\sqrt{30}}7=3.1298431857...$, the largest element of
$M$ corresponding to a sequence of $1$'s and $2$'s without the factor $121$
(indeed $[2; \overline{1,2,2,2}]$ is the largest value of a continued fraction
with coefficients $1$ and $2$ and integer part $2$ without the factor $121$ and
$[0; \overline{2,2,1,2}]$ is the largest value of a continued fraction with
coefficients $1$ and $2$ starting with $[0;2]$ without the factor $121$) and
$[2; \overline{1,1,2}]+[0;\overline{1,1,2}]=\sqrt{10}=3.16227766...$, the
smallest element of $M$ corresponding to a sequence with the factor $121$. The
left endpoint of this interval is also the left endpoint of the plateau, since
if the Markov value of an infinite sequence is strictly smaller than
$\frac{4\sqrt{30}}7$ then, for $n$ large enough, the sequence $(2212)^n$ is
forbidden, which decreases the Hausdorff dimension (the regular Cantor set of
continued fractions with coefficients $1$ and $2$ with the factors $121$ and
$(2212)^n$ forbidden has Hausdorff dimension strictly smaller than the regular
Cantor set where only $121$ is forbidden thanks to Lemma~\ref{l.dimension-drop}).  

On the other hand, the right endpoint of this gap is an isolated point in $M$
(and in $L$). Indeed, if the word $1212$ (or $2121$) is a factor of a sequence,
its Markov value is at least $[2;1,2,\overline{2,1}]+[0;1,\overline{1,2}]>3.28$.
If $12111$ appears as a factor of a sequence, its Markov value is at least
$[2;1,1,1,\overline{1,2}]+[0;1,\overline{1,2}]>3.189$. So, if the Markov value
is smaller than $3.189$ and there is a factor $121$ then it must have a
neighbourhood $2112112$. If the factors $21121122$ and $22112112$ are forbidden
then $121$ forces the neighbourhood $121121121$, which begins and ends by $121$,
and repeating the argument we see that the sequence is forced to be
$\overline{121}$, and the Markov value is the right endpoint of the gap. Let us
determine the smallest possible Markov value of a sequence containing
$2112^*1122$ as a factor (knowing that $1212$, $2121$, $12111$ and $11121$ are
forbidden). It should continue to the left as $22112112^*1122$, and so is
$y^{t}22112^{\#}112^*1122x$, where $x$ and $y$ are infinite sequences. If
$a=[0;1,1,2,2,x]$ and $b=[0;1,1,2,2,y]$, by Lemma \ref{l.plateaux} the Markov
value is minimum when $a=b$ are minimum (with the previous constraints), so when
$a=b=[0;1,1,2,2,\overline{2,1,2,2}]$, and the minimum value is
$[2;1,\overline{1,2,2,2}]+[0;1,1,2,1,\overline{1,2,2,2}]=\frac{16351+720\sqrt{30}}{6409}=3.166578626...$.
This is the right endpoint of the plateau since up to any value strictly larger
than this, we aggregate to the regular Cantor set of continued fractions with
coefficients $1$ and $2$ with the factor $121$ forbidden the finite word
$(2221)^n 121121 (1222)^n$, for $n$ large enough, which connects to any words of
the previous regular Cantor set on both sides, so increasing its Hausdorff
dimension (cf. Lemma~\ref{l.dimension-drop}).  

\subsubsection{The second plateau: first appearances of $212$} 

Cusick identified in \cite{C74} a gap in Markov (and Lagrange) spectrum between $[2; \overline{1,1,1,2,2,2}]+[0; \overline{2,2,1,1,1,2}]=\frac{4\sqrt{210}}{19}=3.050816157...$, the largest element of $M$ corresponding to a sequence of $1$'s and $2$'s without the factors $121$ and $212$ (indeed $[2; \overline{1,1,1,2,2,2}]$ is the largest value of a continued fraction with coefficients $1$ and $2$ and integer part $2$ without the factors $121$ and $212$ and $[0; \overline{2,2,1,1,1,2}]$ is the largest value of a continued fraction with coefficients $1$ and $2$ starting with $[0;2]$ without the factors $121$ and $212$) and $[2; \overline{1,2,2}]+[0;\overline{2,1,2}]=\frac{\sqrt{85}}3=3.07318148576...$, the smallest element of $M$ corresponding to a sequence with the factor $212$. This is the third largest gap before $t_1=3.334384...$, but will produce the second largest plateau, as we will see. The left endpoint of this interval is also the left endpoint of the plateau, since if the Markov value of an infinite sequence is strictly smaller than $\frac{4\sqrt{210}}{19}$ then, for $n$ large enough, the sequence $(221112)^n$ is forbidden, which decreases the Hausdorff dimension (the regular Cantor set of continued fractions with coefficients $1$ and $2$ with the factors $121$, $212$ and $(221112)^n$ forbidden has Hausdorff dimension strictly smaller than the regular Cantor set where only $121$ and $212$ are forbidden, cf. Lemma \ref{l.dimension-drop}). 

On the other hand, the right endpoint of this gap is an isolated point in $M$ (and in $L$). Indeed, if $21222$ appears as a factor of a sequence, its Markov value is at least $[2;1,2,\overline{2,1}]+[0;2,2,\overline{2,1}]>3.1156$. If $212211$ appears as a factor of a sequence and $121$ is forbidden then its Markov value is at least $[2;1,2,\overline{2,1,2,2}]+[0;2,1,1,\overline{1,2,2,2}]>3.0833$. So, if the Markov value is smaller than $3.0833$ and there is a factor $212$ then it must have a neighbourhood $212212212$, which begins and ends by $212$, and repeating the argument we see that the sequence is forced to be $\overline{212}$, and its Markov value is the right endpoint of the gap. Let us determine the smallest possible Markov value of a sequence containing $12212^*211$ as a factor (knowing that $121$, $21222$ and $22212$ are forbidden). It should continue to the left as $112212212^*211$, and so is $y^{t}1122^{\#}12212^*211x$, where $x$ and $y$ are infinite sequences (notice that factors $2212^{o}212$ are not responsible for the Markov value since $[2;1,2,2,\overline{2,1,2,2}]+[0;2,1,2,\overline{2,1,2,2}]<3.078$). If $a=[0;2,1,1,x]$ and $b=[0;2,1,1,y]$, by Lemma \ref{l.plateaux} the Markov value is minimum when $a=b$ are minimum (with the previous constraints), so when $a=b=[0;2,1,1,\overline{1,2,2,2,1,1}]$, and the minimum value is $[2;2,1,1,\overline{1,2,2,2,1,1}]+[0;1,2,2,1,2,2,1,1,\overline{1,2,2,2,1,1}]=\frac{74613+4096\sqrt{210}}{43449}=3.08337773...$. This is the right endpoint of the plateau since up to any value strictly larger than this, we aggregate to the regular Cantor set of continued fractions with coefficients $1$ and $2$ with the factors $121$ and $212$ forbidden the finite word $(222111)^n 22122122 (111222)^n$, for $n$ large enough, which connects to the previous regular Cantor set on both sides, so increasing its Hausdorff dimension (thanks to Lemma \ref{l.dimension-drop}). 

\subsubsection{The third plateau: first appearances of $21222$} 

Cusick identified in~\cite{C74} a gap in Markov (and Lagrange) spectrum between
$[2;1,2,2,\overline{1,1,2,2,2,1}]+[0;2,
\overline{1,1,2,2,2,1}]=\frac{58473-745\sqrt{210}}{15422}=3.09148776579...$, the
largest element of $M$ corresponding to a sequence of $1$'s and $2$'s without
the factors $121$ and $21222$ (indeed $[2;1,2,2,\overline{1,1,2,2,2,1}]$ is the
largest value of a continued fraction with coefficients $1$ and $2$ and integer
part $2$ without the factors $121$ and $21222$ and
$[0;2,\overline{1,1,2,2,2,1}]$ is the largest value of a continued fraction with
coefficients $1$ and $2$ starting with $[0;2]$ without the factors $121$ and
$21222$) and $[2;
\overline{1,2,2,1,2,2,2,2}]+[0;\overline{2,2,2,1,2,2,1,2}]=\frac{\sqrt{233285}}{155}=3.11610231739...$,
the smallest element of $M$ corresponding to a sequence with the factor $21222$.
This is the second largest gap before $t_1=3.334384...$, but produces the third
largest plateau, as we will see. The left endpoint of this interval is also the
left endpoint of the plateau, since if the Markov value of an infinite sequence
is strictly smaller than $\frac{58473-745\sqrt{210}}{15422}$ then, for $n$ large
enough, the sequence $(111222)^n 112212211(222111)^n$ is forbidden, which
decreases the Hausdorff dimension (the regular Cantor set of continued fractions
with coefficients $1$ and $2$ with the factors $121$, $21222$ and $(111222)^n
112212211(222111)^n$ forbidden has Hausdorff dimension strictly smaller than the
regular Cantor set where only $121$ and $21222$ are forbidden). 

On the other hand, the right endpoint of this gap is an isolated point in $M$
(and in $L$). Indeed, if $212221$ appears as a factor of a sequence, its Markov
value is at least $[2;1,2,\overline{2,1}]+[0;2,2,1,\overline{1,2}]>3.1216$. If
$2221222$ appears as a factor of a sequence and $121$ is forbidden then its
Markov value is at least
$[2;1,2,2,2,\overline{2,1,2,2}]+[0;2,2,\overline{2,1,2,2}]>3.1198$. If $2122222$
appears as a factor of a sequence and $121$ is forbidden then its Markov value
is at least $[2;1,2,\overline{2,1,2,2}]+[0;2,2,2,2,\overline{2,1,2,2}]>3.1173$.
If $11221222$ appears as a factor of a sequence and $121$ is forbidden then its
Markov value is at least
$[2;1,2,2,1,1,\overline{1,2,2,2}]+[0;2,2,\overline{2,1,2,2}]>3.117$. If
$21222211$ appears as a factor of a sequence and $121$ is forbidden then its
Markov value is at least
$[2;1,2,\overline{2,1,2,2}]+[0;2,2,2,1,1,\overline{1,2,2,2}]>3.1164$. If
$1221221222$ appears as a factor of a sequence and $121$ is forbidden then its
Markov value is at least
$[2;1,2,2,1,2,2,1,\overline{1,2,2,2}]+[0;2,2,\overline{2,1,2,2}]>3.11611$. So,
if the Markov value is smaller than $3.11611$ and there is a factor $21222$ then
it must have a neighbourhood $2122221221222212$, which begins by $21222$ and
ends by its transpose $22212$, and repeating the argument we see that the
sequence is forced to be $\overline{21222212}$, and its Markov value is the
right endpoint of the gap. Let us determine the smallest possible Markov value
of a sequence containing $12212212^*22$ as a factor (knowing that $121$,
$212221$, $2221222$, $2122222$, $11221222$, $21222211$ and their transposes are
forbidden). It should continue as $12212212^*22212212$, and if it continue as
$12212212^*2221221221$ then it is $x^{t}12212212^*222^{\#}1221221y$, where $x$
and $y$ are infinite sequences. By Lemma \ref{l.plateaux}, the minimum value is
attained when $[0;1,2,2,1,2,2,1,y]=[0;1,2,2,1,2,2,1,x]$ are minimum, so it
should correspond to
$\overline{122211}2212212^*222122122\overline{112221}=\overline{112221}12212212^*2221221221\overline{122211}$
(since $[0,\overline{1,2,2,2,1,1}]$ is the largest continued fraction in $[0,1)$
with coefficients $1, 2$ with $121$ and $122212$ forbidden), and the minimum
value is
\begin{multline*}
    [2;2,2,2,1,2,2,1,2,2,\overline{1,1,2,2,2,1}]+[0;1,2,2,1,2,2,\overline{1,1,2,2,2,1}] \\
    =\frac{1341491634738+14336533\sqrt{210}}{430566005892}=3.1161294028759882...
\end{multline*}
This is the right endpoint of the plateau since up to any value strictly larger
than this, we aggregate to the regular Cantor set of continued fractions with
coefficients $1$ and $2$ with the factors $121$, $212221$, $2221222$, $2122222$,
$11221222$, $11221222$, $1221221222$ and their transposes forbidden the finite
word 
$$(222111)^n 22211221221222212212211222 (111222)^n,$$ for $n$ large enough, which connects to the previous regular Cantor set on both sides, so increasing its Hausdorff dimension. 

\subsubsection{The fourth plateau: first appearances of $12111$} 

Cusick identified in \cite{C74} a gap in Markov (and Lagrange) spectrum between $[2;1,1,\overline{2,2,1,2}]+[0;1,1, \overline{2,2,1,2}]=\frac{15-\sqrt{30}}{3}=3.174258141649...$, the largest element of $M$ corresponding to a sequence of $1$'s and $2$'s without the factors $1212$, $12111$ and their transposes (indeed, after a factor $1,2$, $[0;1,1,\overline{2,2,1,2}]$ is the largest value of a continuation with coefficients $1$ and $2$ and integer part $0$ without the factors $1212$, $12111$ and their transposes) and $[2; \overline{1,1,1,1,2,1,1,2}]+[0;\overline{1,1,2,1,1,1,1,2}]=\frac{\sqrt{9797}}{31}=3.19289664252...$, the smallest element of $M$ corresponding to a sequence with the factor $12111$. This is the fourth largest gap before $t_1=3.334384...$, and produces the fourth largest plateau. The left endpoint of this interval is also the left endpoint of the plateau, since if the Markov value of an infinite sequence is strictly smaller than $\frac{15-\sqrt{30}}{3}$ then, for $n$ large enough, the sequence $(2221)^n 221121122(1222)^n$ is forbidden, which decreases the Hausdorff dimension (the regular Cantor set of continued fractions with coefficients $1$ and $2$ with the factors $1212$, $2121$, $12111$, $11121$ and $(2221)^n 221121122(1222)^n$ forbidden has Hausdorff dimension strictly smaller than the regular Cantor set where only $1212$, $2121$, $12111$ and $11121$ are forbidden). 

On the other hand, the right endpoint of this gap is an isolated point in $M$ (and in $L$). Indeed, if $1212$ appears as a factor of a sequence, its Markov value is at least $[2;1,2,\overline{2,1}]+[0;1,\overline{1,2}]>3.28$. If $121112$ appears as a factor of a sequence, its Markov value is at least $[2;1,1,1,2,\overline{2,1}]+[0;1,\overline{1,2}]>3.207$. If $1112111$ appears as a factor of a sequence, then its Markov value is at least $[2;1,1,1,\overline{1,2}]+[0;1,1,1,\overline{1,2}]>3.22$. If $1211111$ appears as a factor of a sequence and $1212$ and $2121$ are forbidden then its Markov value is at least $[2;1,1,1,1,1,\overline{1,2,1,1}]+[0;1,\overline{1,2,1,1}]>3.197$. If $22112111$ appears as a factor of a sequence and $1212$ and $2121$ are forbidden then its Markov value is at least $[2;1,1,1,\overline{1,2,1,1}]+[0;1,1,2,2,\overline{2,1,2,2}]>3.197$. If $2112112111$ appears as a factor of a sequence and $1212$ and $2121$ are forbidden then its Markov value is at least $[2;1,1,1,\overline{1,2,1,1}]+[0;1,1,2,1,1,2,\overline{2,1,2,2}]>3.1933$. If $12111122$ appears as a factor of a sequence and $1212$ and $2121$ are forbidden then its Markov value is at least $[2;1,1,1,1,2,2,\overline{2,1,2,2}]+[0;1,\overline{1,2,1,1}]>3.19303$. So, if the Markov value is smaller than $3.19303$ and there is a factor $12111$ then it must have a neighbourhood $111211211112112111$, which ends by $12111$ and begins by its transpose $11121$, and repeating the argument we see that the sequence is forced to be $\overline{12111121}$, and its Markov value is the right endpoint of the gap. Let us determine the smallest possible Markov value of a sequence containing $12^*111122$ as a factor (knowing that $1212$, $121112$, $1112111$, $1211111$, $22112111$, $2112112111$ and their transposes are forbidden). It should continue as $211112112^*111122$, and if it continue as $2211112112^*111122$ then it is $x^{t}2211112^{\#}112^*111122y$, where $x$ and $y$ are infinite sequences. By Lemma \ref{l.plateaux}, the minimum value is attained when $[0;1,1,1,1,2,2,y]=[0;1,1,1,1,2,2,x]$ are minimum, so it should correspond to $\overline{2212}2211112112^*111122\overline{2122}=\overline{2221}1112112^*111\overline{1222}$, and the minimum value is $[2;1,1,1,\overline{1,2,2,2}]+[0;1,1,2,1,1,1,\overline{1,2,2,2}]=\frac{332001+3760\sqrt{30}}{110409}=3.1935382818628...$. This is the right endpoint of the plateau since up to any value strictly larger than this, we aggregate to the regular Cantor set of continued fractions with coefficients $1$ and $2$ with the factors $1212$, $121112$, $1112111$, $1211111$, $22112111$, $2112112111$, $12111122$ and their transposes forbidden the finite word $(2221)^n 222111121121111222 (1222)^n$, for $n$ large enough, which connects to the previous regular Cantor set on both sides, so increasing its Hausdorff dimension.

\subsubsection{The fifth plateau: first appearances of $1212$} 

There is a gap in Markov (and Lagrange) spectrum between
$[2;\overline{1,1,1,2}]+[0;\overline{1,1,1,2}]=\frac{4\sqrt{6}}{3}=3.26598632371...$,
the largest element of $M$ corresponding to a sequence of $1$'s and $2$'s
without the factors $1212$, $2121$ (indeed, after a factor $1,2$,
$[0;\overline{1,1,1,2}]$ is the largest value of a continuation with
coefficients $1$ and $2$ and integer part $0$ without the factors $1212$ and
$2121$) and
$[2;\overline{1,2,2,1,2,1,1,2}]+[0;\overline{1,1,2,1,2,2,1,2}]=\frac{\sqrt{689}}{8}=3.2811011871...$,
the smallest element of $M$ corresponding to a sequence with the factor $1212$.
This is the fifth largest gap before $t_1=3.334384...$, and produces the fifth
largest plateau. The left endpoint of this interval is also the left endpoint of
the plateau, since if the Markov value of an infinite sequence is strictly
smaller than $\frac{4\sqrt{6}}{3}$ then, for $n$ large enough, the sequence
$(1112)^n 1112111 (2111)^n$ is forbidden, which decreases the Hausdorff
dimension (the regular Cantor set of continued fractions with coefficients $1$
and $2$ with the factors $1212$, $2121$ and $(1112)^n 1112111 (2111)^n$
forbidden has Hausdorff dimension strictly smaller than the regular Cantor set
where only $1212$ and $2121$ are forbidden). 

On the other hand, the right endpoint of this gap is an isolated point in $M$
(and in $L$). Indeed, if $21212$ appears as a factor of a sequence, its Markov
value is at least $[2;1,2,\overline{2,1}]+[0;1,2,\overline{1,2}]>3.4$. If
$12121$ appears as a factor of a sequence, then its Markov value is at least
$[2;1,2,1,\overline{1,2}]+[0;1,\overline{1,2}]>3.297$. If $212111$ appears as a
factor of a sequence then its Markov value is at least
$[2;1,2,\overline{2,1}]+[0;1,1,1,\overline{1,2}]>3.314$. If $121222$ appears as
a factor of a sequence then its Markov value is at least
$[2;1,\overline{1,2}]+[0;1,2,2,2,\overline{2,1}]>3.2843$. If $2121122$ appears
as a factor of a sequence then its Markov value is at least
$[2;1,2,\overline{2,1}]+[0;1,1,2,2,\overline{2,1}]>3.288$. If $21211211$ appears
as a factor of the sequence then its Markov value is at least
$[2;1,2,\overline{2,1}]+[0;1,1,2,1,1,\overline{1,2}]>3.2826$. If $1212211$
appears as a factor of a sequence then its Markov value is at least
$[2;1,\overline{1,2}]+[0;1,2,2,1,1,\overline{1,2}]>3.2814$. If $12122122$
appears as a factor of a sequence with $12121$, $21212$, $111212$, $212111$,
$222121$ and $121222$ forbidden, then its Markov value is at least
$[2;1,\overline{1,2,1,2,2,1}]+[0;1,2,2,1,2,2,\overline{2,1,2,2}]>3.28121$. So,
if the Markov value is smaller than $3.28121$ and there is a factor $1212$ then
it must have a neighbourhood $212112122121$, which begins and ends by $2121$,
and repeating the argument we see that the sequence is forced to be
$\overline{12122121}$, and its Markov value is the right endpoint of the gap.
Let us determine the smallest possible Markov value of a sequence containing
$12^*122122$ as a factor (knowing that $21212$, $12121$, $111212$, $2121122$,
$21211211$, $121222$ , $1212211$ and their transposes are forbidden). It should
continue as $212212112^*122122$, and if it continues as $2212212112^*122122$ it
is $x^{t}2212212^{\#}112^*122122y$, where $x$ and $y$ are infinite sequences. By
Lemma \ref{l.plateaux}, the minimum value is attained when
$[0;1,2,2,1,2,2,y]=[0;1,2,2,1,2,2,x]$ are minimum, so it should correspond to
$\overline{2212}2212212112^*122122\overline{2122}=\overline{2122}12112^*1\overline{2212}$,
and the minimum value is
$[2;1,\overline{2,2,1,2}]+[0;1,1,2,1,\overline{2,2,1,2}]=\frac{18879+800\sqrt{30}}{7089}=3.28124988856557...$.
This is the right endpoint of the plateau since up to any value strictly larger
than this, we aggregate to the regular Cantor set of continued fractions with
coefficients $1$ and $2$ with the factors $12121$, $21212$, $111212$, $2121122$,
$21211211$, $121222$, $1212211$ and their transposes forbidden the finite word
$(2221)^n 222122121121222 (1222)^n$, for $n$ large enough, which connects to the
previous regular Cantor set on both sides, so increasing its Hausdorff
dimension. 

\subsubsection{The sixth plateau: first appearances of $111222$} We have a gap in Markov (and Lagrange) spectrum between $[2;2,1,1,\overline{2,2,2,1,1,2,2,1,1,1,2,2,1,1}]+[0;1,1,1,\overline{2,2,1,1,2,2,2,1,1,2,2,1,1,1}]=\frac{28\sqrt{213378}}{4287}=3.017028188796...$, the largest element of $M$ corresponding to a sequence of $1$'s and $2$'s without the factors $121$, $212$, $111222$ and $222111$ (indeed, if the Markov value of a sequence with center $2112^*22$ is attained at $0$, then it is at most $[2;2,2,\overline{1,1,1,2,2,2}]+[0;1,1,2,2,\overline{1,1,1,2,2,2}]<3.007$; also, $[2;2,1,1,\overline{2,2,2,1,1,2,2,1,1,1,2,2,1,1}]$ is the largest value of a continued fraction with coefficients $1$ and $2$ starting with $[2;2,1]$ without the factors $121$, $212$, $111222$ and $222111$ and $[0;1,1,1,\overline{2,2,1,1,2,2,2,1,1,2,2,1,1,1}]$ is the largest value of a continued fraction with coefficients $1$ and $2$ starting with $[0;1,1]$ without the factors $121$, $212$, $111222$ and $222111$) and $[2; \overline{2,2,2,1,1,1,1,2}]+[0;\overline{1,1,1,1,2,2,2,2}]=\frac{\sqrt{229}}5=3.02654919...$, the smallest element of $M$ corresponding to a sequence with the factor $111222$. The left endpoint of this interval is also the left endpoint of the plateau, since if the Markov value of an infinite sequence is strictly smaller than $\frac{28\sqrt{213378}}{4287}$ then, for $n$ large enough, the sequence $(22211221112211)^n$ is forbidden, which decreases the Hausdorff dimension (the regular Cantor set of continued fractions with coefficients $1$ and $2$ with the factors $121$, $212$, $111222$, $222111$ and $(22211221112211)^n$ forbidden has Hausdorff dimension strictly smaller than the regular Cantor set where only $121$, $212$, $111222$ and $222111$ are forbidden). 

On the other hand, the right endpoint of this gap is an isolated point in $M$ (and in $L$). Indeed, if the word $1112221$ is a factor of a sequence, its Markov value is at least $[2;1,\overline{1,1,1,2,2,2}]+[0;2,2,1,1,\overline{2,2,2,1,1,1}]>3.0322$. If $2221112$ appears as a factor of a sequence, its Markov value is at least $[2;1,1,1,2,2,\overline{1,1,1,2,2,2}]+[0;\overline{2,2,2,1,1,1}]>3.0433$. If $22211111$ appears as a factor of a sequence, its Markov value is at least $[2;1,1,1,1,1,\overline{1,2,2,2,1,1}]+[0;\overline{2,2,2,1,1,1}]>3.0305$. If $22222111$ appears as a factor of a sequence, its Markov value is at least $$[2;1,1,1,\overline{1,2,2,2,1,1}]+[0;2,2,2,2,\overline{2,1,1,1,2,2}]>3.02742.$$ 
So, if the Markov value is smaller than $3.02742$ and there is a factor $222111$ then it must have a neighbourhood $112222111122$. If $2221111221$ appears as a factor of such a sequence, its Markov value is at least 
$$[2;1,1,1,1,2,2,1,\overline{1,2,2,2,1,1,2,2,1,1,1,2,2,1}]+[0;2,2,2,1,1,1,1,\overline{2,2,1,1,1,2,2,1,1,2,2,2,1,1}]>3.0266664.$$ If $2112222111$ appears as a factor of a sequence, its Markov value is at least 
$$[2;1,1,1,1,2,2,2,2,1,1,\overline{2,2,2,1,1,2,2,1,1,1,2,2,1,1}]+[0;2,2,2,1,1,2,\overline{2,1,1,1,2,2,1,1,2,2,2,1,1,2}]>3.0266663.$$ 
If the factors $2221111221$ and $2112222111$ (and their transposes) are
forbidden then $222111$ forces the neighbourhood $1111222211112222$, and
repeating the argument we see that the sequence is forced to be
$\overline{11112222}$, and the Markov value is the right endpoint of the gap.
Let us determine the smallest possible Markov value of a sequence containing
$2112222^*111$ as a factor (knowing that $121$, $212$, $1222111$, $2111222$,
$11111222$, $11122222$, $2221111221$ and their transposes are forbidden). It
should continue as $2112222^*11112^{\#}222112$, and so is
$y^{t}2112222^*11112^{\#}222112x$, where $x$ and $y$ are infinite sequences. If
$a=[0;2,2,2,1,1,2,x]$ and $b=[0;2,2,2,1,1,2,y]$, by Lemma \ref{l.plateaux} the
Markov value is minimum when $a=b$ are minimum (with the previous constraints),
so when $a=b=[0;2,2,2,1,1,2,\overline{2,1,1,1,2,2,1,1,2,2,2,1,1,2}]$, and the
minimum value is 
\begin{multline*}
    [2;\overline{2,2,2,1,1,2,2,1,1,1,2,2,1,1}]+[0;1,1,1,1,2,\overline{2,2,2,1,1,2,2,1,1,1,2,2,1,1}]
    \\     =\frac{2497149255+2842763 \sqrt{213378}}{1258910718} =3.026666336477...
\end{multline*}
This is the right endpoint of the plateau since up to any value strictly larger than this, we aggregate to the regular Cantor set of continued fractions with coefficients $1$ and $2$ with the factors $121$, $212$, $1222111$, $2111222$, $11111222$, $11122222$, $2221111221$, $2112222111$ and transposes forbidden the finite word $(22211221112211)^n 222211112222 (11221112211222)^n$, for $n$ large enough, which connects to the previous regular Cantor set on both sides, so increasing its Hausdorff dimension.

\subsubsection{The seventh plateau: first appearances of $211121112$} We have a gap in Markov (and Lagrange) spectrum between $[2;\overline{1,1,1,1,1,2,1,1,1,2}]+[0;\overline{1,1,1,2,1,1,1,1,1,2}]=\frac{\sqrt{66045}}{79}=3.25306604786...$, the largest element of $M$ corresponding to a sequence of $1$'s and $2$'s without the factors $1212$, $2121$ and $211121112$ (indeed, $[2;\overline{1,1,1,1,1,2,1,1,1,2}]$ is the largest value of a continued fraction with coefficients $1$ and $2$ starting with $[2;1,1,1,1]$ without the factors $1212$, $2121$ and $211121112$ and $[0;\overline{1,1,1,2,1,1,1,1,1,2}]$ is the largest value of a continued fraction with coefficients $1$ and $2$ preceded by $12$ and starting with $0$ without the factors $1212$, $2121$ and $211121112$) and $[2;1,1,1,2,\overline{2,1,2,2}]+[0;1,1,1,2,\overline{2,1,2,2}]=\frac{190+4\sqrt{30}}{65}=3.2601369584647...$, the smallest element of $M$ corresponding to a sequence with the factor $211121112$. The left endpoint of this interval is also the left endpoint of the plateau, since if the Markov value of an infinite sequence is strictly smaller than $\frac{\sqrt{66045}}{79}$ then, for $n$ large enough, the sequence $(2111211111)^n$ is forbidden, which decreases the Hausdorff dimension (the regular Cantor set of continued fractions with coefficients $1$ and $2$ with the factors $1212$, $2121$, $211121112$ and $(2111211111)^n$ forbidden has Hausdorff dimension strictly smaller than the regular Cantor set where only $1212$, $2121$ and $211121112$ are forbidden, cf. Lemma \ref{l.dimension-drop}). 

On the other hand, the right endpoint of this gap is $[2;1,1,1,2,\overline{2,1,2,2}]+[0;1,1,1,2,\overline{2,1,2,2}]$ since $[0;1,1,1,2,\overline{2,1,2,2}]$ is the smaller value of a continued fraction with coefficients $1$ and $2$ starting with $[0;1,1,1,2]$ without the factors $1212$ and $2121$. This is the right endpoint of the plateau since up to any value strictly larger than this, we aggregate to the regular Cantor set of continued fractions with coefficients $1$ and $2$ with the factors $1212$, $2121$ and $211121112$ forbidden the finite word $(2212)^n 211121112 (2122)^n$, for $n$ large enough, which connects to the previous regular Cantor set on both sides, so increasing its Hausdorff dimension (thanks to Lemma \ref{l.dimension-drop}).

\subsubsection{The eighth plateau: first appearances of $111212$} We have a gap in Markov (and Lagrange) spectrum between $[2;\overline{1,1,2,2,1,2}]+[0;\overline{1,2,1,1,2,2}]=\frac{11+35\sqrt{87}}{102}=3.30841438096...$, the largest element of $M$ corresponding to a sequence of $1$'s and $2$'s without the factors $21212$, $111212$ and $212111$ (indeed, $[2;\overline{1,1,2,2,1,2}]$ is the largest value of a continued fraction with coefficients $1$ and $2$ starting with $[2;1,1,2]$ without the factors $21212$, $212111$and $111212$ and $[0;\overline{1,2,1,1,2,2}]$ is the largest value of a continued fraction with coefficients $1$ and $2$ starting with $[0;1,2,1,1]$ without the factors $21212$, $111212$ and $212111$) and $[2;\overline{1,2,2,1,2,1,1,1,1,2}]+[0;\overline{1,1,1,1,2,1,2,2,1,2}]=\frac{\sqrt{33245}}{55}=3.31512935564...$, the smallest element of $M$ corresponding to a sequence with the factor $111212$. The left endpoint of this interval is also the left endpoint of the plateau, since if the Markov value of an infinite sequence is strictly smaller than $\frac{11+35\sqrt{87}}{102}$ then, for $n$ large enough, the sequence $(112122)^n 1121211 (221211)^n$ is forbidden, which decreases the Hausdorff dimension (the regular Cantor set of continued fractions with coefficients $1$ and $2$ with the factors $21212$, $212111$, $111212$ and $(112122)^n 1121211 (221211)^n$ forbidden has Hausdorff dimension strictly smaller than the regular Cantor set where only $21212$, $212111$ and $111212$ are forbidden). 

On the other hand, the right endpoint of this gap is an isolated point in $M$ (and in $L$). Indeed, if the word $2121112$ is a factor of a sequence, its Markov value is at least $[2;1,1,1,2,\overline{2,1}]+[0;1,2,\overline{2,1}]>3.3329$. If $21211111$ appears as a factor of a sequence, its Markov value is at least $[2;1,1,1,1,1,\overline{1,2}]+[0;1,2,\overline{2,1}]>3.32$. If $1212111$ appears as a factor of a sequence, its Markov value is at least $[2;1,1,1,\overline{1,2}]+[0;1,2,1,\overline{1,2}]>3.33$. If $22212111$ appears as a factor of a sequence, its Markov value is at least $[2;1,1,1,\overline{1,2}]+[0;1,2,2,2,\overline{2,1}]>3.3189$. If $212111122$ appears as a factor of a sequence, its Markov value is at least $[2;1,1,1,1,2,2,\overline{2,1,2,1,1,1}]+[0;1,2,\overline{2,1,2,1,1,1}]>3.3162$. If $112212111$ appears as a factor of a sequence, its Markov value is at least $[2;1,1,1,\overline{1,2,1,2,1,1}]+[0;1,2,2,1,1,\overline{1,2,1,2,1,1}]>3.3161$. If $2212212111$ appears as a factor of a sequence, its Markov value is at least $[2;1,1,1,1,2,1,2,2,1,\overline{1,2,1,2,1,1}]+[0;1,2,2,1,2,2,\overline{2,1,2,1,1,1}]>3.31523$. If $2121111211$ appears as a factor of a sequence, its Markov value is at least $[2;1,1,1,1,2,1,1,\overline{1,2,1,1,1,2}]+[0;1,2,\overline{2,1,2,1,1,1}]>3.31537$. If $211212212111$ appears as a factor of a sequence, its Markov value is at least $[2;1,1,1,1,2,1,2,2,1,\overline{1,2,1,2,1,1,1,1}]+[0;1,2,2,1,2,1,1,2,\overline{2,1,2,1,1,1,1,1}]>3.315145$. So, if the Markov value is smaller than $3.315145$ and there is a factor $212111$ then it must have a neighbourhood $1112122121111212$, and repeating the argument we see that the sequence is forced to be $\overline{2122121111}$, and the Markov value is the right endpoint of the gap. Let us determine the smallest possible Markov value of a sequence containing $211212212^*111$ as a factor (knowing that $21212$, $2121112$, $21211111$, $1212111$, $22212111$, $212111122$, $112212111$, $2212212111$, $2121111211$ and their transposes are forbidden). It should continue as $211212212^*11112^{\#}122121$, and so is $y^{t}211212212^*11112^{\#}12212112x$, where $x$ and $y$ are infinite sequences. If $a=[0;1,2,2,1,2,1,1,2,x]$ and $b=[0;1,2,2,2,1,2,1,1,2,y]$, by  Lemma \ref{l.plateaux} the Markov value is minimum when $a=b$ are minimum (with the previous constraints), so when $a=b=[0;1,2,2,1,2,1,1,2,\overline{2,1,2,1,1,2}]$, and the minimum value is $[2;\overline{1,2,2,1,2,1}]+[0;1,1,1,1,2,\overline{1,2,2,1,2,1}]=\frac{56508+2716\sqrt{87}}{24687}=3.3151521654...$. This is the right endpoint of the plateau since up to any value strictly larger than this, we aggregate to the regular Cantor set of continued fractions with coefficients $1$ and $2$ with the factors $21212$, $2121112$, $21211111$, $1212111$, $22212111$, $212111122$, $112212111$, $2212212111$, $2121111211$, $211212212111$ and transposes forbidden the finite word $(221121)^n 221211112122 (121122)^n$, for $n$ large enough, which connects to the previous regular Cantor set on both sides, so increasing its Hausdorff dimension.
\vskip.1in

\subsubsection{The ninth plateau: first appearances of $1112111$} We have a gap
in Markov (and Lagrange) spectrum between
$[2;1,1,1,2,1,1,2,2,\overline{1,2,2,2}]+[0;1,1,2,2,\overline{1,2,2,2}]=\frac{14775-599\sqrt{30}}{3570}=3.21964758558657...$,
the largest element of $M$ corresponding to a sequence of $1$'s and $2$'s
without the factors $1212$, $2121$ and $1112111$ (indeed
$[2;1,1,1,2,1,1,2,2,\overline{1,2,2,2}]$ is the largest value of a continued
fraction with coefficients $1$ and $2$ starting with $[2;1,1]$ without the
factors $1212$, $2121$ and $1112111$ and $[0;1,1,2,2,\overline{1,2,2,2}]$ is the
largest value of a continued fraction with coefficients $1$ and $2$ starting
with $[0;1,1,2]$ without the factors $1212$, $2121$ and $1112111$) and
$[2;\overline{1,1,1,1,2}]+[0;\overline{1,1,1,1,2}]=\frac{2\sqrt{65}}5=3.2249030993...$,
the smallest element of $M$ corresponding to a sequence with the factor
$1112111$. The left endpoint of this interval is also the left endpoint of the
plateau, since if the Markov value of an infinite sequence is strictly smaller
than $\frac{14775-599\sqrt{30}}{3570}$ then, for $n$ large enough, the sequence
$(2221)^n 2211211121122(1222)^n$ is forbidden, which decreases the Hausdorff
dimension (the regular Cantor set of continued fractions with coefficients $1$
and $2$ with the factors $1212$, $2121$, $1112111$ and $(2221)^n
2211211121122(1222)^n$ forbidden has Hausdorff dimension strictly smaller than
the regular Cantor set where only $1212$, $2121$ and $1112111$ are forbidden). 

On the other hand, the right endpoint of this gap is an isolated point in $M$ (and in $L$). Indeed, if the word $11121112$ is a factor of a sequence, its Markov value is at least $[2;1,1,1,2,\overline{2,1,2,2}]+[0;1,1,1,\overline{1,2,1,1}]>3.24$. If $111112111$ appears as a factor of a sequence, its Markov value is at least $[2;1,1,1,\overline{1,2,1,1}]+[0;1,1,1,1,1,\overline{1,2,1,1}]>3.229$. If $2211112111$ appears as a factor of a sequence, its Markov value is at least $[2;1,1,1,\overline{1,2,1,1}]+[0;1,1,1,1,2,2,\overline{2,1,2,2}]>3.2256$. If $211211112111$ appears as a factor of a sequence, its Markov value is at least 
$$[2;1,1,1,\overline{1,2,1,1}]+[0;1,1,1,1,2,1,1,2,\overline{2,1,2,2}]>3.22492.$$
So, if the Markov value is smaller than $3.22492$ and there is a factor
$1112111$ then it must have a neighbourhood $11121111211112111$, which begins
and ends by $1112111$, and repeating the argument we see that the sequence is
forced to be $\overline{21111}$, and the Markov value is the right endpoint of
the gap. Let us determine the smallest possible Markov value of a sequence
containing $211211112^*111$ as a factor (knowing that $11121112$, $111112111$,
$2211112111$ and their transpose are forbidden). It should continue as
$211211112^*1111211112112$, and so is $y^{t}211211112^{*}11112^{\#}11112112x$,
where $x$ and $y$ are infinite sequences. If $a=[0;1,1,1,1,2,1,1,2,x]$ and
$b=[0;1,1,1,1,2,1,1,2,y]$, by Lemma \ref{l.plateaux} the Markov value is minimum
when $a=b$ are minimum (with the previous constraints), so when
$a=b=[0;1,1,1,1,2,1,1,2,\overline{2,1,2,2}]$, and the minimum value is
\begin{multline*}
[2;1,1,1,1,2,1,1,2,\overline{2,1,2,2}]+[0;1,1,1,1,2,1,1,1,1,2,1,1,2,\overline{2,1,2,2}]\\
=\frac{9537392579+1230099\sqrt{30}}{2959418678}=3.22500164632....
\end{multline*}
This is the right endpoint of the plateau since up to any value strictly larger
than this, we aggregate to the regular Cantor set of continued fractions with
coefficients $1$ and $2$ with the factors $11121112$, $111112111$, $2211112111$,
$211211112111$ and their transpose forbidden the finite word $(2221)^n
221121111211112111121122(1222)^n$, for $n$ large enough, which connects to the
previous regular Cantor set on both sides, so increasing its Hausdorff
dimension. 

\subsubsection{The tenth plateau: first appearances of $21112$} We have a gap in Markov (and Lagrange) spectrum between $[2;\overline{1,1,1,1,1,2,2,1,1,2,2,2,1,1,2,2}]+[0;2,\overline{1,1,2,2,2,1,1,2,2,1,1,1,1,1,2,2}]=\frac{10\sqrt{718341}}{2819}=3.006562605623...$, the largest element of $M$ corresponding to a sequence of $1$'s and $2$'s without the factors $121$, $212$, $111222$, $222111$ and $21112$ (indeed, if the Markov value of a sequence with center $222^*112$ without the factors $121$, $212$, $111222$, $222111$ is attained at $0$, then it is at most $[2;2,2,1,1,2,2,\overline{1,1,1,2,2,2}]+[0;1,1,2,2,\overline{1,1,1,2,2,2}]<3.0059$; also, $[2;\overline{1,1,1,1,1,2,2,1,1,2,2,2,1,1,2,2}]$ is the largest value of a continued fraction with coefficients $1$ and $2$ with integer part $2$ without the factors $121$, $212$, $111222$, $222111$ and $21112$ and $[0;2,\overline{1,1,2,2,2,1,1,2,2,1,1,1,1,1,2,2}]$ is the largest value of a continued fraction with coefficients $1$ and $2$ starting with $[0;2,1]$ without the factors $121$, $212$, $111222$, $222111$ and $21112$) and $[2;\overline{1,1,1,2,2}]+[0;\overline{2,1,1,1,2}]=\frac{\sqrt{145}}4=3.010398644698...$, the smallest element of $M$ corresponding to a sequence with the factor $21112$. The left endpoint of this interval is also the left endpoint of the plateau, since if the Markov value of an infinite sequence is strictly smaller than $\frac{10\sqrt{718341}}{2819}$ then, for $n$ large enough, the sequence $(1111122112221122)^n$ is forbidden, which decreases the Hausdorff dimension (the regular Cantor set of continued fractions with coefficients $1$ and $2$ with the factors $121$, $212$, $111222$, $222111$, $21112$ and $(1111122112221122)^n$ forbidden has Hausdorff dimension strictly smaller than the regular Cantor set where only $121$, $212$, $111222$, $222111$ and $21112$ are forbidden, cf. Lemma \ref{l.dimension-drop}). 

On the other hand, the right endpoint of this gap is an isolated point in $M$ (and in $L$). Indeed, if the word $211221112$ is a factor of a sequence, its Markov value is at least $[2;1,1,1,2,\overline{2,1,1,1,2,2}]+[0;2,1,1,2,\overline{2,1,1,1,2,2}]>3.016$. If $1111221112$ appears as a factor of a sequence, its Markov value is at least $$[2;1,1,1,2,\overline{2,1,1,1,2,2}]+[0;2,1,1,1,1,1,\overline{2,2,2,1,1,1}]>3.0118.$$ So, if the Markov value is smaller than $3.0118$ and there is a factor $21112$ then it must have a neighbourhood $211122111221112$, which begins and ends by $21112$, and repeating the argument we see that the sequence is forced to be $\overline{21112}$, and the Markov value is the right endpoint of the gap. Let us determine the smallest possible Markov value of a sequence containing $111122^*1112$ as a factor (knowing that $121$, $212$, $222111$, $211221112$ and their transpose are forbidden). It should continue as $111122^*11122111221111$, and so is $y^{t}111122^*111221112^{\#}21111x$, where $x$ and $y$ are infinite sequences. If $a=[0;2,1,1,1,1,x]$ and $b=[0;2,1,1,1,1,y]$, by Lemma \ref{l.plateaux} the Markov value is minimum when $a=b$ are minimum (with the previous constraints), so when $a=b=[0;2,\overline{1,1,1,1,1,2,2,1,1,2,2,2,1,1,2,2}]$, and the minimum value is $[2;2,\overline{1,1,1,1,1,2,2,1,1,2,2,2,1,1,2,2}]+[0;1,1,1,2,2,1,1,1,2,2,\overline{1,1,1,1,1,2,2,1,1,2,2,2,1,1,2,2}]=\frac{14264157401+16294182 \sqrt{718341}}{9321104530}=3.011906071937...$. This is the right endpoint of the plateau since up to any value strictly larger than this, we aggregate to the regular Cantor set of continued fractions with coefficients $1$ and $2$ with the factors $121$, $212$, $222111$, $211221112$, $1111221112$ and their transpose forbidden the finite word 
$$(1111122112221122)^n 1111122111221112211111(2211222112211111)^n,$$ for $n$ large enough, which connects to the previous regular Cantor set, so increasing its Hausdorff dimension.

\subsubsection{The eleventh plateau: first appearances of $21112221$} We have a gap in Markov (and Lagrange) spectrum between $[2;\overline{1,1,1,2,2,2,2,2}]+[0;\overline{2,2,2,2,1,1,1,2}]=\frac{8\sqrt{1785}}{111}=3.0449917368328...$, the largest element of $M$ corresponding to a sequence of $1$'s and $2$'s without the factors $121$, $212$, $12221112$ and $21112221$ (indeed, $[2;\overline{1,1,1,2,2,2,2,2}]$ is the largest value of a continued fraction with coefficients $1$ and $2$ without the factors $121$, $212$, $12221112$ and $21112221$ and $[0;\overline{2,2,2,2,1,1,1,2}]$ is the largest value of a continued fraction with coefficients $1$ and $2$ starting with $[0;2,2,2]$ without the factors $121$, $212$, $12221112$ and $21112221$), and 
$$[2;\overline{2,2,1,1,2,2,2,1,1,1,2,2,1,1,1,2}]+[0;\overline{1,1,1,2,2,1,1,1,2,2,2,1,1,2,2,2}]=\frac{5\sqrt{2059061}}{2353}=3.0491773395772939...,$$ the smallest element of $M$ corresponding to a sequence with the factor $21112221$. The left endpoint of this interval is also the left endpoint of the plateau, since if the Markov value of an infinite sequence is strictly smaller than $\frac{8\sqrt{1785}}{111}$ then, for $n$ large enough, the sequence $(22222111)^n 22(11122222)^n$ is forbidden, which decreases the Hausdorff dimension (the regular Cantor set of continued fractions with coefficients $1$ and $2$ with the factors $121$, $212$, $12221112$, $21112221$ and $(22222111)^n 22(11122222)^n$ forbidden has Hausdorff dimension strictly smaller than the regular Cantor set where only $121$, $212$, $12221112$ and $21112221$ are forbidden). 

On the other hand, the right endpoint of this gap is an isolated point in $M$ (and in $L$). Indeed, if $1222111222$ appears as a factor of a sequence, its Markov value is at least $[2;1,1,1,2,2,2,2,\overline{1,1,1,2,2,2}]+[0;2,2,1,1,\overline{2,2,2,1,1,1}]>3.0495$. If $1112221112$ appears as a factor of a sequence, then its Markov value is at least $[2;1,1,1,2,2,\overline{1,1,1,2,2,2}]+[0;2,2,1,1,1,1,\overline{2,2,2,1,1,1}]>3.0498$. If $122211122112$ appears as a factor of a sequence, its Markov value is at least $[2;1,1,1,2,2,1,1,2,2,\overline{1,1,1,2,2,2}]+[0;2,2,1,1,\overline{2,2,2,1,1,1}]>3.0492$. If $122112221112$ appears as a factor of a sequence, then its Markov value is at least $[2;1,1,1,2,2,\overline{1,1,1,2,2,2}]+[0;2,2,1,1,2,2,1,1,\overline{2,2,2,1,1,1}]>3.0492$. If $1222111221111$ appears as a factor of a sequence, its Markov value is at least $[2;1,1,1,2,2,1,1,\overline{1,1,1,2,2,2}]+[0;2,2,1,1,\overline{2,2,2,1,1,1}]>3.04919$. If $2222112221112$ appears as a factor of a sequence, then its Markov value is at least $[2;1,1,1,2,2,\overline{1,1,1,2,2,2}]+[0;2,2,1,1,2,2,\overline{2,2,2,1,1,1}]>3.04918$. If $122211122111221$ appears as a factor of a sequence, its Markov value is at least $$[2;1,1,1,2,2,1,1,1,2,2,1,1,2,2,2,\overline{1,1,1,2,2,2,2,2}]+[0;2,2,1,1,2,2,2,\overline{1,1,1,2,2,2,2,2}]>3.0491779.$$ If $211222112221112$ appears as a factor of a sequence, then its Markov value is at least $[2;1,1,1,2,2,1,1,1,\overline{2,2,2,1,1,1,1,1}]+[0;2,2,1,1,2,2,2,1,1,2,2,\overline{1,1,1,2,2,2}]>3.049178$. If $1111222112221112$ appears as a factor of a sequence, then its Markov value is at least $[2;1,1,1,2,2,1,1,1,\overline{2,2,2,1,1,1,1,1}]+[0;2,2,1,1,2,2,2,1,1,1,1,1,2,2,2,\overline{1,1,1,2,2,2,2,2}]>3.0491775$. If $1222111221112222$ appears as a factor of a sequence, its Markov value is at least $$[2;1,1,1,2,2,1,1,1,2,2,2,2,2,\overline{1,1,1,2,2,2,2,2}]+[0;2,2,1,1,2,2,2,\overline{1,1,1,2,2,2,2,2}]>3.0491774.$$ So, if the Markov value is smaller than $3.0491774$ and there is a factor $21112221$ then it must have a neighbourhood $211122211222111221112221$, which begins and ends by $21112221$, and repeating the argument we see that the sequence is forced to be $\overline{2111222112221112}$, and its Markov value is the right endpoint of the gap. Let us determine the smallest possible Markov value of a sequence containing $1222^*111221112222$ as a factor (knowing that $212$, $121$, $1222111222$, $1112221112$, $122211122112$, $122112221112$, $1222111221111$, $2222112221112$, $122211122111221$, $211222112221112$, $1111222112221112$ and their transposes are forbidden). It should continue as $2221112211122211222^*111221112222$, and if it continues as $22221112211122211222^*111221112222$ it is $x^{t}2222111221112^{\#}2211222^*111221112222y$, where $x$ and $y$ are infinite sequences. By Lemma \ref{l.plateaux}, the minimum value is attained when $[0;1,1,1,2,2,1,1,1,2,2,2,2,y]=[0;1,1,1,2,2,1,1,1,2,2,2,2,x]$ are minimum, so it should correspond to 
$$\overline{22222111}222221112211122211222^*1112211122222\overline{11122222},$$ and the minimum value is 
\begin{eqnarray*}[2;1,1,1,2,2,\overline{1,1,1,2,2,2,2,2}]+[0;2,2,1,1,2,2,2,1,1,1,2,2,\overline{1,1,1,2,2,2,2,2}]&=&\frac{89042158285+148687392\sqrt{1785}}{31262231449} \\ &=&3.049177432380786...
\end{eqnarray*} 
This is the right endpoint of the plateau since up to any value strictly larger than this, we aggregate to the regular Cantor set of continued fractions with coefficients $1$ and $2$ with the factors $212$, $121$, $1222111222$, $1112221112$, $122211122112$, $122112221112$, $1222111221111$, $2222112221112$, $122211122111221$, $211222112221112$, $1111222112221112$, $1222111221112222$ and their transposes forbidden the finite word $(22222111)^n 221112221122211122 (11122222)^n$, for $n$ large enough, which connects to the previous regular Cantor set on both sides, so increasing its Hausdorff dimension. 

\subsubsection{The twelfth plateau: first appearances of $12121122$} We have a gap in Markov (and Lagrange) spectrum between $[2;\overline{1,1,2,1,1,2,1,2}]+[0;\overline{1,2,1,1,2,1,1,2}]=\frac{24\sqrt{35}}{43}=3.3019980184742...$, the largest element of $M$ corresponding to a sequence of $1$'s and $2$'s without the factors $21212$, $111212$, $212111$, $12121122$ and $22112121$ (indeed, $[2;\overline{1,1,2,1,1,2,1,2}]$ is the largest value of a continued fraction with coefficients $1$ and $2$ starting with $[2;1,1,2,1]$ without the factors $21212$, $111212$, $212111$, $12121122$ and $22112121$ and $[0;\overline{1,2,1,1,2,1,1,2}]$ is the largest value of a continued fraction with coefficients $1$ and $2$ starting with $[0;1,2,1,1,2,1]$ without the factors $21212$, $111212$, $212111$, $12121122$ and $22112121$) and $[2;1,2,1,1,2,1,2,1,1,2,2,\overline{2,1,2,1,1,2}]+[0;1,1,2,2,\overline{2,1,2,1,1,2}]=\frac{457878845-1713407 \sqrt{87}}{133663998}=3.306030457347...$, the smallest element of $M$ corresponding to a sequence with the factor $12121122$. The left endpoint of this interval is also the left endpoint of the plateau, since if the Markov value of an infinite sequence is strictly smaller than $\frac{24\sqrt{35}}{43}$ then, for $n$ large enough, the sequence $(11211212)^n$ is forbidden, which decreases the Hausdorff dimension (the regular Cantor set of continued fractions with coefficients $1$ and $2$ with the factors $21212$, $111212$, $212111$, $12121122$, $22112121$ and $(11211212)^n$ forbidden has Hausdorff dimension strictly smaller than the regular Cantor set where only $21212$, $111212$, $212111$, $12121122$ and $22112121$ are forbidden). 

On the other hand, the right endpoint of this gap is $[2;1,2,1,1,2,1,2,1,1,2,2,\overline{2,1,2,1,1,2}]+[0;1,1,2,2,\overline{2,1,2,1,1,2}]$ since
the minimum Markov value of a sequence with center $1212^*1122$ is attained at a sequence of the type $x^t22112^{\#}1211212^*1122y$, where $x$ and $y$ are infinite sequences. By Lemma \ref{l.plateaux}, the minimum value is attained when $[0;1,1,2,2,y]=[0;1,1,2,2,x]$ are minimum, so it should correspond to $\overline{211212}221121211212^*1122\overline{212112}$, and the minimum value is $[2;1,1,2,2,\overline{2,1,2,1,1,2}]+[0;1,2,1,1,2,1,2,1,1,2,2,\overline{2,1,2,1,1,2}]=\frac{457878845-1713407\sqrt{87}}{133663998}=3.3060304573471179...$. It coincides with the right endpoint of the plateau since up to any value strictly larger than this, we aggregate to the regular Cantor set of continued fractions with coefficients $1$ and $2$ with the factors $21212$, $111212$, $212111$, $12121122$ and $22112121$ forbidden the finite word $(221121)^n 222112121121211222 (121122)^n$, for $n$ large enough, which connects to the previous regular Cantor set on both sides, so increasing its Hausdorff dimension.

\subsection{Between the largest plateaux: end of proof of Theorem~\ref{t.A} modulo dimension estimates at checkpoints}
\label{ss:inbetween}
For two sets $A, B \subset \mathbb R$ we say that $A \prec B$ if $x<y$ for all
$x\in A$ and $y\in B$. We would like now to list the twelve plateaux of~$d(t)$ in order of
their \emph{occurrence}
on the real line (from left to right) 
$$
P_0 \prec P_{11} \prec P_7 \prec P_{12} \prec P_3 \prec P_4 \prec P_2 \prec P_5 \prec
P_{10} \prec P_8 \prec P_6 \prec P_{13} \prec P_9 \prec P_1. 
$$ 
In order to show that $P_i$, $2\leq i\leq 11$, are indeed the~$10$ largest plateaux we provide estimates
on the values of the dimension function~$d$ on the intervals in between. 
To setup the computation, we choose the intervals between the two plateaux
$P_i = (a_i,b_i) \prec P_j=(a_j, b_j)$ given by 
$$
\intpx{i}{j}=[10^{-4}([10^4 b_{i}]-1),10^{-4} ([10^4 a_j]+1)],
$$
where $[x]$ stands for the integer part of the number~$x$. 
We obtain the intervals $\intpx{i}{j}$ by straightforward
computation: it gives, for example, $ \intpx{0}{11} =  [2.9999,3.0066] $, 
$ \intpx{11}{7} =  [3.0117,3.0171] $, and so on. 
\pvshort{In Appendix~\ref{ap:table} we list lower and upper bounds on the
dimension function at certain ``checkpoints'', chosen at the distance $0.0024$
apart (with a few exceptions), to show that there is no other plateaux larger
than~$0.005$. 
This completes the proof of Theorem~\ref{t.A}. }

In the last~\S\ref{s.approx} we give an algorithm for computing the dimension
function~$d$. 

\section{Numerical method}
\label{s.approx}
\pvlong{Computation of the dimension function~$d$ is done by obtaining lower and upper
bounds on the values of~$d$ at a large number of points $s_k \in (3, t_1)$ and
is based on the following two observations~\cite[\S2.1]{MMPV}. 
\\
\indent First, we fix the maximal length of forbidden strings~$n$, say, $n=16$ (larger
value of~$n$ will result in a more accurate approximation of~$d$). 
Then for a point $s \in (a,b)$ we build recursively a finite set of finite
``forbidden'' sequences $\beta_{-n} \beta_{-n+1} \ldots \beta_{-1} \beta_0
\beta_1 \ldots \beta_n$ so that any infinite extension $\alpha \in \{1,
2\}^{\mathbb  Z} $ with the property that $\alpha_j = \beta_j$ for all  $-n \le j
\le n$ we have that  
$
\lambda_0(\alpha) = [0; \alpha_{-1}, \alpha_{-2}, \ldots] + [\alpha_0; \alpha_1, \alpha_2, \ldots] > s.
$
\\
\indent Clearly, after excluding from $E_2$ all irrational numbers whose continued 
fraction expansion contains a ``forbidden'' string, we obtain a Cantor set $K
\subsetneq E_2$ such that $M \cap (-\infty, s) \subset 2 + K + K$.
Recall that $\dim (K + K) \le \dim K + \dim_B K$, where $\dim_B K$ denotes the upper box dimension.
It is known that $\dim_B K = \dim K$ for these types of sets (cf. Chapter 4 of
Palis--Takens book~\cite{PT}) and hence an upper bound on~$\dim K$ gives us an
upper bound 
$$
d(s) = \dim((-\infty,s) \cap M  ) \le 2\dim K.
$$
Once the set of forbidden sequences is constructed, we may compute the maximal 
Markov value of strings which do not contain a substring from~$K$ and denote it
by~$s^\prime$.
It was shown in~\cite[proof of Lemma~3]{Mor18} that the dimension of~$K$ gives us a lower
bound on $d(s^\prime)$ via the inequality
$$
\min\{2 \dim K, 1\} \le \dim (( -\infty, s^\prime) \cap M ) = d(s^\prime) . 
$$
Repeating the procedure described above for~$n$ sufficiently large and for a
sequence of values~$\{s_k\} \in (a,b) $, we obtain an
upper bound for $d(s_k)$ and a lower bound for $d(s_k^\prime)$. Monotonicity
of~$d$ implies that $\max_{s_k^\prime < s} d(s_k^\prime)$ is a lower bound
for~$d(s)$. This allows us to obtain an accurate approximation to the
dimension function. The pseudocode for the recursive function that constructs the set of forbidden
words is given in Algorithm~\ref{alg:forbid} below.}

\begin{algorithm}
    \caption{Recursive computation of the set of forbidden strings.}\label{alg:forbid}
    \begin{algorithmic}[1]
        \Function{ForbiddenWords}{$\beta$, $lmax$, $mmax$, $s$, $n$ }
       
        \Require{
 $s, n$: threshold and the maximal length of the string;
 $\beta$: the current string; 
 $mmax$: the current largest Markov number of the strings that do not contain
        forbidden substrings;
        $lmax$: the current maximal length of the forbidden strings.}
       
        \If{ $|\beta| > 0$ } 
        
        \Let{$(a,b)$}{convex hull of $\lambda_0$ obtained from continuations of
        $\beta$;}

        \If{ $a > s$ }             
            \State Save $\beta$ and its reverse; \Comment{Markov values are larger than~$s$} 
            \If{ $lmax < |\beta|$ } \Let{$lmax$}{$|\beta|$}  \EndIf
        \ElsIf{ $b < s$ } \Comment{Markov values are smaller than~$s$} 
        \If{ $b > mmax$ } \Let{$mmax$}{$b$} \EndIf
        \ElsIf{ $|\beta| < 2n+1 $} 
        \If{ $|\beta|$ is even } \Comment{add one symbol on the right} 
               \State \Call{ForbiddenWords}{$\beta1$, $lmax+1$, $mmax$, $s$, $n$ }
               \State \Call{ForbiddenWords}{$\beta2$, $lmax+1$, $mmax$, $s$, $n$ }
               \Else \Comment{add one symbol on the left} 
               \State \Call{ForbiddenWords}{$1\beta$, $lmax+1$, $mmax$, $s$, $n$ }
               \State \Call{ForbiddenWords}{$2\beta$, $lmax+1$, $mmax$, $s$, $n$ }
            \EndIf
            \ElsIf{ $b > mmax$} \Comment{$|\beta|=2n+1$ and $s$ is in the convex
            hull of~$\lambda_0$.} 
            \Let{$mmax$}{$b$} 
        \EndIf
        \Else \Comment{$\beta$ is the empty word}
        \State \Call{ForbiddenWords}{$1$, $1$, $0$, $s$, $n$ }
        \State \Call{ForbiddenWords}{$2$, $1$, $0$, $s$, $n$ }
        \EndIf
        \EndFunction

    \end{algorithmic}
\end{algorithm}

\pvlong{In order to confirm that all continuations of a certain string result in a
Markov value greater than the chosen threshold, the following 
fact~\cite[Lemma~2.4]{MMPV} can be used. 
\\
\begin{lemma}
    \label{lem:intlim}
    For any sequence $\alpha \in \{1,2\}^{\mathbb{Z}}$ we have an upper bound
    $$
    \lambda_0(\alpha) \le 
    \begin{cases}
    [\alpha_0; \alpha_1 \dots \alpha_j \overline{12}] + 
    [0;\alpha_{-1} \dots \alpha_{-k} \overline{12}], &\mbox{ if } k \mbox{ and }
    j \mbox{ are even,} \\
    [\alpha_0; \alpha_1 \dots \alpha_j \overline{21}] + 
    [0;\alpha_{-1} \dots \alpha_{-k} \overline{12}], &\mbox{ if } k \mbox{ is
    even and } j \mbox{ is odd,} \\ 
    [\alpha_0; \alpha_1 \dots \alpha_j \overline{12}] + 
    [0;\alpha_{-1} \dots \alpha_{-k} \overline{21}], &\mbox{ if } k \mbox{ is
    odd and } j \mbox{ is even,} \\ 
    [\alpha_0; \alpha_1 \dots \alpha_j \overline{21}] + 
    [0;\alpha_{-1} \dots \alpha_{-k} \overline{21}], &\mbox{ if } k \mbox{ and }
    j \mbox{ are odd;} 
    \end{cases}
    $$
    and a lower bound 
    $$
    \lambda_0(\alpha) \ge 
    \begin{cases}
    [\alpha_0; \alpha_1 \dots \alpha_j \overline{21}] + 
    [0;\alpha_{-1} \dots \alpha_{-k} \overline{21}], &\mbox{ if } k \mbox{ and }
    j \mbox{ are even,} \\
    [\alpha_0; \alpha_1 \dots \alpha_j \overline{12}] + 
    [0;\alpha_{-1} \dots \alpha_{-k} \overline{21}], &\mbox{ if } k \mbox{ is
    even and } j \mbox{ is odd,} \\ 
    [\alpha_0; \alpha_1 \dots \alpha_j \overline{21}] + 
    [0;\alpha_{-1} \dots \alpha_{-k} \overline{12}], &\mbox{ if } k \mbox{ is
    odd and } j \mbox{ is even,} \\ 
    [\alpha_0; \alpha_1 \dots \alpha_j \overline{12}] + 
    [0;\alpha_{-1} \dots \alpha_{-k} \overline{12}], &\mbox{ if } k \mbox{ and }
    j \mbox{ are odd.} 
    \end{cases}
    $$
\end{lemma}
We now want to illustrate the recursive procedure for constructing the set of forbidden
strings using. 
\begin{example}
Let us fix $n=6$ and consider a point $s=3.2658$. Then we take 
$\beta_0 = 1$ and compute lower and upper bounds on $\lambda_0$ using
Lemma~\ref{lem:intlim}, which gives us $1.73 < \lambda_0(\beta) < 2.47$. Since both values less than~$s$, we make a record of the upper
bound, and consider $\beta_0 = 2$. This choice gives us $2.73 <
\lambda_0(\beta) < 3.47$, so we consider continuation $\beta_{-1}\beta_0 = 12$ that
gives $2.94 < \lambda_0(\beta) < 3.47$ and then add $\beta_1 = 1$ that gives $3.15
< \lambda_0(\beta) < 3.47$. We proceed, adding at odd steps a digit on
the left and at even steps a digit on the right, until either we get an 
interval for~$\lambda_0$ that lies to the right of~$s$ or to the left of~$s$, 
or we obtain a string of length $2n+1$ continuations of which contain~$s$.
\\
\begin{minipage}{50mm}
\begin{tabular}{c|rl}
  $|\beta|$ & $\beta$ \\
  \hline
$ 13 $ &  $1   1   2   1   1   1   2^* $ & $\kern-11pt   1   1   1   2   1   2$  \\
$ 13 $ &  $2   1   2   1   1   1   2^* $ & $\kern-11pt   1   1   1   2   1   1$  \\
$ 12 $ &  $2   1   2   1   1   1   2^* $ & $\kern-11pt  1   1   1   2   1$  \\
$ 12 $ &  $1   2   1   1   1   2^* $ & $\kern-11pt  1   1   1   2   1   2$  \\
$  5 $ &  $1   1   2^* $ & $\kern-11pt  1   2$  \\
$  5 $ &  $2   1   2^* $ & $\kern-11pt  1   1$  \\
$  4 $ &  $2   1   2^* $ & $\kern-11pt  1$  \\
$  4 $ &  $1   2^* $ & $\kern-11pt  1   2$  \\
  \hline
\end{tabular}
\end{minipage}%
\begin{minipage}{110mm}
    \noindent For $s=3.2658$ and $n=6$ the function in Algorithm~\ref{alg:forbid} gives~$8$ forbidden words, that are displayed on the right in the order they
have been constructed. It also gives a maximum Markov value for non-excluded
strings~$s^\prime = 3.2660$. We see that the first~$6$ strings are actually redundant, since
they contain the strings of length~$4$ as substrings. After further reduction we
are left with only~$2$ forbidden strings of length~$4$, $2121$ and~$1212$. The
algorithm developed in~\cite{PV20} gives an estimate $\dim K \le 0.46933$. We
therefore conclude that $d(3.2658) \le 0.93866 \le d(3.2660)$. 
\end{minipage}
\end{example}}

\appendix

\section{Estimates on the dimension function at checkpoints}
\label{ap:table}




\begin{longtable}{|c|c|c|c|c|}

  \hline
   Plateau or an  & Length & Checkpoints & \multicolumn{2}{c|}{Hausdorff dimension} \\\hhline{~|~|~|-|-}
   intermediate interval & of the plateau & in the interval & lower bound & upper bound \\\hline
  \endfirsthead
  \multicolumn{5}{@{}l}{\ldots continued}\\\hline
    Plateau or an  & Length & Checkpoints & \multicolumn{2}{c|}{Hausdorff
    dimension} \\
    \hhline{~|~|~|-|-}
   intermediate interval & of the plateau & in the interval & lower bound &
   upper bound \\
    \hline
\endhead
 \multirow{3}{*}{  $ \intpx{0}{11} =  [2.9999,3.0066] $} 
  &  &  $3.0015 $  &  $ 0.432861 $ & $ 0.432861 $ \\
  &  &  $3.0031 $  &  $ 0.455261 $ & $ 0.455261 $ \\
  &  &  $3.0055 $  &  $ 0.510071 $ & $ 0.510071 $ \\
    \hline
        $P_{11}$ &  $0.00534347$ & & \multicolumn{2}{c|}{0.536334}   \\
    \hline
   \multirow{2}{*}{$ \intpx{11}{7} =  [3.0117,3.0171] $ } 
  &  &  $3.0126 $  &  $ 0.545441 $ & $ 0.545441 $ \\
  &  &  $3.0150 $  &  $ 0.551178 $ & $ 0.551178 $ \\
    \hline
        $P_7$ &  $0.00963815$ & & \multicolumn{2}{c|}{0.569770}  \\
    \hline
  \multirow{8}{*}{ $ \intpx{7}{12} =  [3.0265,3.0451] $  } 
  &  &  $3.0273 $  &  $ 0.577393 $ & $ 0.577393 $ \\
  &  &  $3.0297 $  &  $ 0.589935 $ & $ 0.589935 $ \\
  &  &  $3.0321 $  &  $ 0.617218 $ & $ 0.620728 $ \\
  &  &  $3.0345 $  &  $ 0.659302 $ & $ 0.659302 $ \\
  &  &  $3.0369 $  &  $ 0.662750 $ & $ 0.665924 $ \\
  &  &  $3.0393 $  &  $ 0.686493 $ & $ 0.686493 $ \\
  &  &  $3.0417 $  &  $ 0.692932 $ & $ 0.692932 $ \\
  &  &  $3.0441 $  &  $ 0.700073 $ & $ 0.700073 $ \\
   \hline 
       $P_{12}$ &  $0.0041857$  &  & \multicolumn{2}{c|}{0.709914}  \\
  \hline
  $ \intpx{12}{3} =  [3.0490,3.0509] $  & \multicolumn{4}{c|}{No estimate needed:
  $|\intpx{12}{3}| = 0.0019 < 0.005$. \phantom{\Large $R_{1}^2$} } \\
  \hline
       $P_3$ & $0.03256158$ & & \multicolumn{2}{c|}{0.728108}  \\
  \hline
  \multirow{3}{*}{ $ \intpx{3}{4}  =  [3.0832,3.0915]  $ } 
  &  &  $3.0856 $  &  $ 0.730499 $ & $ 0.730499 $ \\
  &  &  $3.0880 $  &  $ 0.739929 $ & $ 0.739929 $ \\
  &  &  $3.0904 $  &  $ 0.742004 $ & $ 0.742004 $ \\
  \hline
          $P_4$  & $0.02464164$ & & \multicolumn{2}{c|}{0.750628} \\
  \hline
  \multirow{5}{*}{ $ \intpx{4}{2} = [3.1160,3.1299]$ } 
  &  &  $3.1184 $  &  $ 0.765778 $ & $ 0.766235 $ \\
  &  &  $3.1208 $  &  $ 0.775299 $ & $ 0.775299 $ \\
  &  &  $3.1232 $  &  $ 0.783356 $ & $ 0.784485 $ \\
  &  &  $3.1256 $  &  $ 0.796478 $ & $ 0.796478 $ \\
  &  &  $3.1280 $  &  $ 0.807770 $ & $ 0.807770 $ \\
  \hline
          $P_2$ & $0.03673544$ & & \multicolumn{2}{c|}{0.812150}  \\
  \hline
  \multirow{3}{*}{ $ \intpx{2}{5} = [3.1664,3.1743] $ } 
  &  &  $3.1672 $  &  $ 0.812927 $ & $ 0.813019 $ \\
  &  &  $3.1696 $  &  $ 0.815735 $ & $ 0.815735 $ \\
  &  &  $3.1720 $  &  $ 0.818542 $ & $ 0.818542 $ \\
  \hline
        $P_5$ & $0.01928014$ & & \multicolumn{2}{c|}{0.827194}  \\
    \hline
  \multirow{11}{*}{  $ \intpx{5}{10} =  [3.1934,3.2170] $ } 
  &  &  $3.1950 $  &  $ 0.830841 $ & $ 0.830841 $ \\
  &  &  $3.1974 $  &  $ 0.830841 $ & $ 0.831421 $ \\
  &  &  $3.1998 $  &  $ 0.843140 $ & $ 0.846527 $ \\
  &  &  $3.2022 $  &  $ 0.851990 $ & $ 0.851990 $ \\
  &  &  $3.2046 $  &  $ 0.859192 $ & $ 0.861786 $ \\
  &  &  $3.2070 $  &  $ 0.868286 $ & $ 0.869415 $ \\
  &  &  $3.2094 $  &  $ 0.869629 $ & $ 0.869629 $ \\
  &  &  $3.2118 $  &  $ 0.876587 $ & $ 0.877350 $ \\
  &  &  $3.2142 $  &  $ 0.878082 $ & $ 0.878082 $ \\
  &  &  $3.2166 $  &  $ 0.882294 $ & $ 0.883118 $ \\
  &  &  $3.2190 $  &  $ 0.888153 $ & $ 0.888794 $ \\ 
  \hline
         $P_{10}$ & $0.00535406$ & & \multicolumn{2}{c|}{0.889660}  \\
   \hline
 \multirow{12}{*}{  $ \intpx{10}{8} =  [3.2248,3.2531]     $  }  
  &  &  $3.2265 $  &  $ 0.890961 $ & $ 0.891022 $ \\
  &  &  $3.2289 $  &  $ 0.892151 $ & $ 0.892151 $ \\
  &  &  $3.2313 $  &  $ 0.895325 $ & $ 0.897522 $ \\
  &  &  $3.2337 $  &  $ 0.901459 $ & $ 0.901825 $ \\
  &  &  $3.2361 $  &  $ 0.903961 $ & $ 0.905518 $ \\
  &  &  $3.2385 $  &  $ 0.907532 $ & $ 0.909607 $ \\
  &  &  $3.2409 $  &  $ 0.910797 $ & $ 0.910797 $ \\
  &  &  $3.2433 $  &  $ 0.912079 $ & $ 0.913177 $ \\
  &  &  $3.2457 $  &  $ 0.917023 $ & $ 0.917480 $ \\
  &  &  $3.2481 $  &  $ 0.919891 $ & $ 0.921661 $ \\
  &  &  $3.2505 $  &  $ 0.926117 $ & $ 0.928192 $ \\
  &  &  $3.2529 $  &  $ 0.931702 $ & $ 0.932281 $ \\
   \hline
      $P_{8}$  &  $0.00707091$ &  & \multicolumn{2}{c|}{0.932262}  \\
   \hline
   \multirow{2}{*}{ $ \intpx{8}{6}  =  [3.2600,3.2660] $ } 
  &  &  $3.2624 $  &  $ 0.934113 $ & $ 0.934113 $ \\
  &  &  $3.2648 $  &  $ 0.937103 $ & $ 0.937103 $ \\
   \hline
   $P_{6}$ &  $0.01526356 $ &  & \multicolumn{2}{c|}{0.938646} \\
 \hline
   \multirow{9}{*}{ $ \intpx{6}{13} =  [3.2811,3.3021] $ }
  &  &  $3.2827 $  &  $ 0.938995 $ & $ 0.939240 $ \\
  &  &  $3.2851 $  &  $ 0.944000 $ & $ 0.945496 $ \\
  &  &  $3.2875 $  &  $ 0.949036 $ & $ 0.950256 $ \\
  &  &  $3.2899 $  &  $ 0.954529 $ & $ 0.956268 $ \\
  &  &  $3.2923 $  &  $ 0.960846 $ & $ 0.960846 $ \\
  &  &  $3.2947 $  &  $ 0.965393 $ & $ 0.965973 $ \\
  &  &  $3.2971 $  &  $ 0.966095 $ & $ 0.966095 $ \\
  &  &  $3.2995 $  &  $ 0.966187 $ & $ 0.966187 $ \\
  &  &  $3.3019 $  &  $ 0.967072 $ & $ 0.967834 $ \\
 \hline
   $P_{13}$ & $0.00403244 $ &  & \multicolumn{2}{c|}{0.967812} \\
 \hline
 $ \intpx{13}{9} =  [3.3059,3.3085] $ & \multicolumn{4}{c|}{No estimate needed:
  $|\intpx{13}{9}| = 0.0036 < 0.005$. \phantom{\Large $R_{1}^2$} } \\
 \hline
   $P_{9}$ &  $0.00673778 $ &  & \multicolumn{2}{c|}{0.971588} \\
 \hline
   \multirow{8}{*}{ $ \intpx{9}{1} = [3.3150,3.3344] $ } 
   &  &  $3.3158 $  &  $  0.971619 $  &  $  0.972382 $ \\
   &  &  $3.3182 $  &  $  0.976440 $  &  $  0.976929 $ \\
   &  &  $3.3206 $  &  $  0.979279 $  &  $  0.980682 $ \\
   &  &  $3.3230 $  &  $  0.985535 $  &  $  0.986542 $ \\
   &  &  $3.3254 $  &  $  0.990051 $  &  $  0.991913 $ \\
   &  &  $3.3278 $  &  $  0.994781 $  &  $  0.995270 $ \\
   &  &  $3.3302 $  &  $  0.995544 $  &  $  0.995575 $ \\
   &  &  $3.3326 $  &  $  0.995605 $  &  $  0.995605 $ \\
 \hline
\end{longtable}

\newpage

\section{Plot of the dimension function}
\label{ap:plot}

\vspace*{-7mm}
\begin{figure}[h!]
\caption{A sketch of the graph of the dimension function~$d$ on the intervals
between the plateaux. The blue curve
corresponds to the upper bound, the red curve corresponds to the lower bound.
Estimates are obtained using a mesh of size~$0.00005$. We would like to point
out that all plots are not to scale and, moreover, scale on axis is different on
every plot. }
\centering
\begin{subfigure}{0.28\textwidth}
  \includegraphics{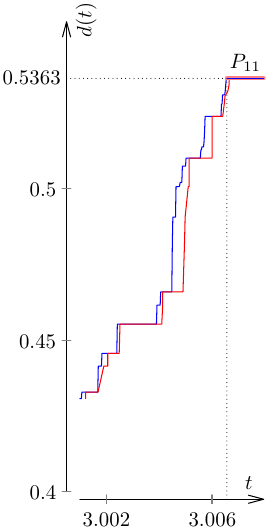}
  \subcaption{ $t \in \intpx{0}{11}$} 
  \label{fig:int011}
\end{subfigure}
\begin{subfigure}{0.2\textwidth}
  \includegraphics{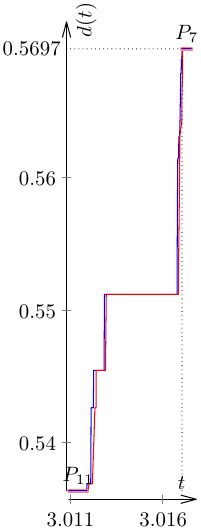}
  \subcaption{ $t \in \intpx{11}{7}$} 
  \label{fig:int117}
\end{subfigure}
\begin{subfigure}{0.28\textwidth}
  \includegraphics{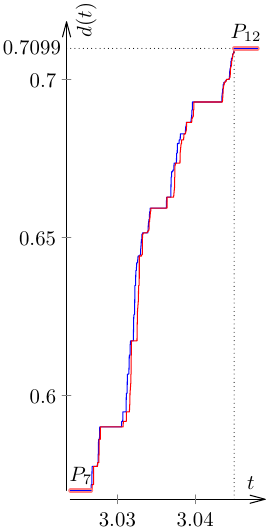}
  \subcaption{ $t \in \intpx{7}{12}$} 
  \label{fig:int712}
\end{subfigure} 
\begin{subfigure}{0.2\textwidth}
  \includegraphics{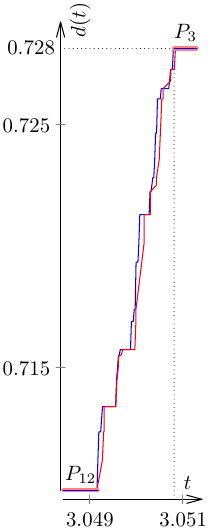}
  \subcaption{ $t \in \intpx{12}{3}$} 
  \label{fig:int123}
\end{subfigure} \\ 
\begin{subfigure}{0.3\textwidth}
  \includegraphics{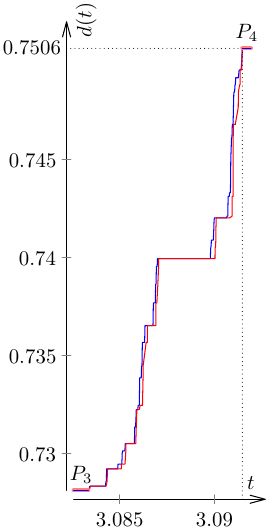}
  \subcaption{ $t \in \intpx{3}{4}$} 
  \label{fig:int34}
\end{subfigure}
\begin{subfigure}{0.3\textwidth}
  \includegraphics{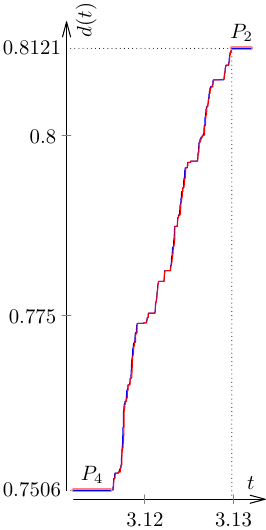}
  \subcaption{ $t \in \intpx{4}{2}$} 
  \label{fig:int42}
\end{subfigure}
\begin{subfigure}{0.3\textwidth}
  \includegraphics{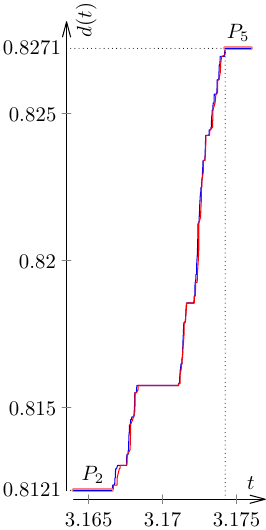}
  \subcaption{ $t \in \intpx{2}{5}$} 
  \label{fig:int25}
\end{subfigure}
\end{figure}%
\begin{figure}\ContinuedFloat
\begin{subfigure}{0.3\textwidth}
  \includegraphics{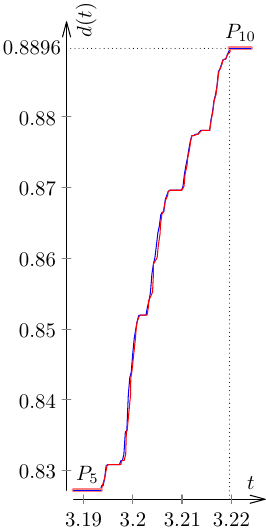}
  \subcaption{ $t \in \intpx{5}{10}$} 
  \label{fig:int510}
\end{subfigure}
\begin{subfigure}{0.3\textwidth}
  \includegraphics{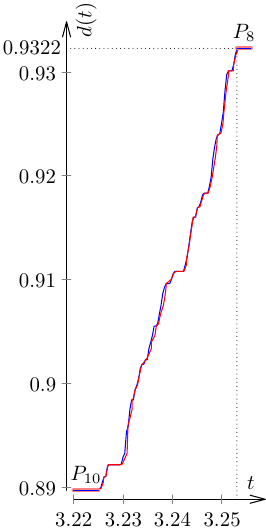}
  \subcaption{ $t \in \intpx{10}{8}$} 
  \label{fig:int108}
\end{subfigure}
\begin{subfigure}{0.3\textwidth}
  \includegraphics{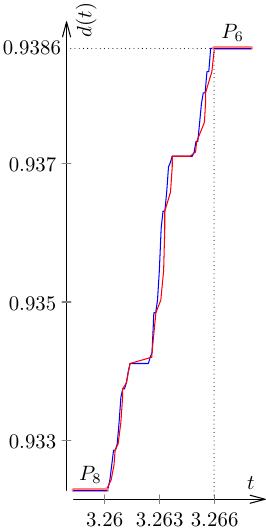}
  \subcaption{ $t \in \intpx{8}{6}$} 
  \label{fig:int68}
\end{subfigure}
\end{figure}
\begin{figure}\ContinuedFloat
\begin{subfigure}{0.3\textwidth}
  \includegraphics{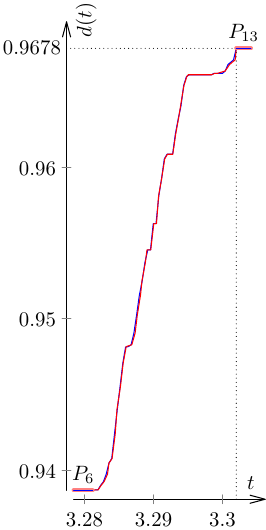}
  \subcaption{ $t \in \intpx{6}{13}$} 
  \label{fig:int613}
\end{subfigure}
\begin{subfigure}{0.3\textwidth}
  \includegraphics{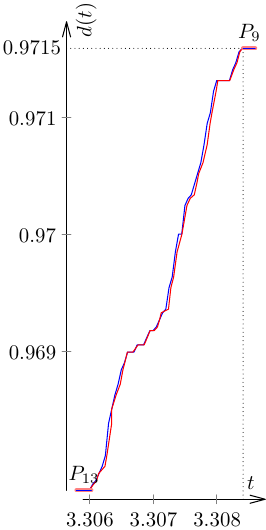}
  \subcaption{ $t \in \intpx{13}{9}$} 
  \label{fig:int139}
\end{subfigure}
\begin{subfigure}{0.3\textwidth}
  \includegraphics{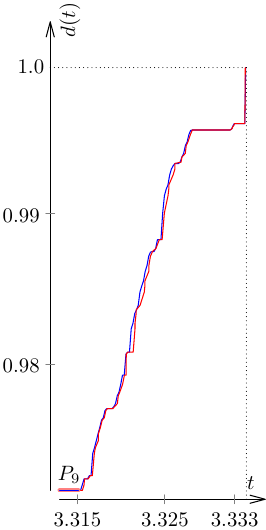}
  \subcaption{ $t \in \intpx{9}{1}$} 
  \label{fig:int91}
\end{subfigure}
\end{figure}

\end{document}